\documentclass[12pt,reqno]{smfart}
\usepackage{amsmath,amssymb,smfthm,stmaryrd,epic,enumerate}
\usepackage[frenchb]{babel}
\usepackage[latin1]{inputenc}
\usepackage{a4wide}
\usepackage{graphicx}

\usepackage{pstricks}

\usepackage{xypic}
\usepackage[active]{srcltx}

\usepackage{lmodern}

\usepackage{hyperref}

\newcommand*{\htarrow}{\lhook\joinrel\relbar\joinrel\twoheadrightarrow}
\newcommand*{\tarrow}{\relbar\joinrel\mid\joinrel\twoheadrightarrow}
\newcommand*{\harrow}{\lhook\joinrel\relbar\joinrel\mid\joinrel\rightarrow}
\newcommand*{\rarrow}{\relbar\joinrel\mid\joinrel\rightarrow}

\NumberTheoremsIn{subsection}\SwapTheoremNumbers

\begin{document}
\newtheorem{prop-defi}[smfthm]{Proposition-DÈfinition}
\newtheorem{notas}[smfthm]{Notations}
\newtheorem{nota}[smfthm]{Notation}
\newtheorem{defis}[smfthm]{DÈfinitions}
\newtheorem{hypo}[smfthm]{HypothËse}

\def\Xm{{\mathbb X}}
\def\Um{{\mathbb U}}
\def\Am{{\mathbb A}}
\def\Fm{{\mathbb F}}
\def\Mm{{\mathbb M}}
\def\Nm{{\mathbb N}}
\def\Pm{{\mathbb P}}
\def\Qm{{\mathbb Q}}
\def\Zm{{\mathbb Z}}
\def\Dm{{\mathbb D}}
\def\Cm{{\mathbb C}}
\def\Rm{{\mathbb R}}
\def\Gm{{\mathbb G}}
\def\Lm{{\mathbb L}}
\def\Km{{\mathbb K}}
\def\Om{{\mathbb O}}
\def\Em{{\mathbb E}}

\def\BC{{\mathcal B}}
\def\QC{{\mathcal Q}}
\def\TC{{\mathcal T}}
\def\ZC{{\mathcal Z}}
\def\AC{{\mathcal A}}
\def\CC{{\mathcal C}}
\def\DC{{\mathcal D}}
\def\EC{{\mathcal E}}
\def\FC{{\mathcal F}}
\def\GC{{\mathcal G}}
\def\HC{{\mathcal H}}
\def\IC{{\mathcal I}}
\def\JC{{\mathcal J}}
\def\KC{{\mathcal K}}
\def\LC{{\mathcal L}}
\def\MC{{\mathcal M}}
\def\NC{{\mathcal N}}
\def\OC{{\mathcal O}}
\def\PC{{\mathcal P}}
\def\UC{{\mathcal U}}
\def\VC{{\mathcal V}}
\def\XC{{\mathcal X}}
\def\SC{{\mathcal S}}

\def\BF{{\mathfrak B}}
\def\AF{{\mathfrak A}}
\def\GF{{\mathfrak G}}
\def\EF{{\mathfrak E}}
\def\CF{{\mathfrak C}}
\def\DF{{\mathfrak D}}
\def\JF{{\mathfrak J}}
\def\LF{{\mathfrak L}}
\def\MF{{\mathfrak M}}
\def\NF{{\mathfrak N}}
\def\XF{{\mathfrak X}}
\def\UF{{\mathfrak U}}
\def\KF{{\mathfrak K}}
\def\FF{{\mathfrak F}}

\def \longmapright#1{\smash{\mathop{\longrightarrow}\limits^{#1}}}
\def \mapright#1{\smash{\mathop{\rightarrow}\limits^{#1}}}
\def \lexp#1#2{\kern \scriptspace \vphantom{#2}^{#1}\kern-\scriptspace#2}
\def \linf#1#2{\kern \scriptspace \vphantom{#2}_{#1}\kern-\scriptspace#2}
\def \linexp#1#2#3 {\kern \scriptspace{#3}_{#1}^{#2} \kern-\scriptspace #3}

\def \Ext{\mathop{\mathrm{Ext}}\nolimits}
\def \ad{\mathop{\mathrm{ad}}\nolimits}
\def \sh{\mathop{\mathrm{Sh}}\nolimits}
\def \irr{\mathop{\mathrm{Irr}}\nolimits}
\def \FH{\mathop{\mathrm{FH}}\nolimits}
\def \FPH{\mathop{\mathrm{FPH}}\nolimits}
\def \coh{\mathop{\mathrm{Coh}}\nolimits}
\def \res{\mathop{\mathrm{res}}\nolimits}
\def \op{\mathop{\mathrm{op}}\nolimits}
\def \rec {\mathop{\mathrm{rec}}\nolimits}
\def \art{\mathop{\mathrm{Art}}\nolimits}
\def \hyp {\mathop{\mathrm{Hyp}}\nolimits}
\def \cusp {\mathop{\mathrm{Cusp}}\nolimits}
\def \scusp {\mathop{\mathrm{Scusp}}\nolimits}
\def \Iw {\mathop{\mathrm{Iw}}\nolimits}
\def \JL {\mathop{\mathrm{JL}}\nolimits}
\def \speh {\mathop{\mathrm{Speh}}\nolimits}
\def \isom {\mathop{\mathrm{Isom}}\nolimits}
\def \Vect {\mathop{\mathrm{Vect}}\nolimits}
\def \groth {\mathop{\mathrm{Groth}}\nolimits}
\def \hom {\mathop{\mathrm{Hom}}\nolimits}
\def \deg {\mathop{\mathrm{deg}}\nolimits}
\def \val {\mathop{\mathrm{val}}\nolimits}
\def \det {\mathop{\mathrm{det}}\nolimits}
\def \rep {\mathop{\mathrm{Rep}}\nolimits}
\def \spec {\mathop{\mathrm{Spec}}\nolimits}
\def \fr {\mathop{\mathrm{Fr}}\nolimits}
\def \frob {\mathop{\mathrm{Frob}}\nolimits}
\def \ker {\mathop{\mathrm{Ker}}\nolimits}
\def \im {\mathop{\mathrm{Im}}\nolimits}
\def \Red {\mathop{\mathrm{Red}}\nolimits}
\def \red {\mathop{\mathrm{red}}\nolimits}
\def \aut {\mathop{\mathrm{Aut}}\nolimits}
\def \diag {\mathop{\mathrm{diag}}\nolimits}
\def \spf {\mathop{\mathrm{Spf}}\nolimits}
\def \Def {\mathop{\mathrm{Def}}\nolimits}
\def \twist {\mathop{\mathrm{Twist}}\nolimits}
\def \supp {\mathop{\mathrm{Supp}}\nolimits}
\def \Id {{\mathop{\mathrm{Id}}\nolimits}}
\def \lie {{\mathop{\mathrm{Lie}}\nolimits}}
\def \Ind{\mathop{\mathrm{Ind}}\nolimits}
\def \ind {\mathop{\mathrm{ind}}\nolimits}
\def \soc {\mathop{\mathrm{Soc}}\nolimits}
\def \top {\mathop{\mathrm{Top}}\nolimits}
\def \ker {\mathop{\mathrm{Ker}}\nolimits}
\def \coker {\mathop{\mathrm{Coker}}\nolimits}
\def \coim {\mathop{\mathrm{Coim}}\nolimits}
\def \gal {{\mathop{\mathrm{Gal}}\nolimits}}
\def \Nr {{\mathop{\mathrm{Nr}}\nolimits}}
\def \rn {{\mathop{\mathrm{rn}}\nolimits}}
\def \tr {{\mathop{\mathrm{Tr~}}\nolimits}}
\def \Sp {{\mathop{\mathrm{Sp}}\nolimits}}
\def \st {{\mathop{\mathrm{St}}\nolimits}}
\def \sp{{\mathop{\mathrm{Sp}}\nolimits}}
\def \perv{\mathop{\mathrm{Perv}}\nolimits}
\def \tor {{\mathop{\mathrm{Tor}}\nolimits}}
\def \nrd {{\mathop{\mathrm{Nrd}}\nolimits}}
\def \nilp {{\mathop{\mathrm{Nilp}}\nolimits}}
\def \obj {{\mathop{\mathrm{Obj}}\nolimits}}
\def \cl {{\mathop{\mathrm{cl}}\nolimits}}
\def \can {{\mathop{\mathrm{can}}\nolimits}}
\def \gr {{\mathop{\mathrm{gr}}\nolimits}}
\def \grr {{\mathop{\mathrm{grr}}\nolimits}}
\def \cogr {{\mathop{\mathrm{gr}}\nolimits}}
\def \cogrr {{\mathop{\mathrm{grr}}\nolimits}}

\def \rem{{\noindent\textit{Remarque:~}}}
\def \ext {{\mathop{\mathrm{Ext}}\nolimits}}
\def \End {{\mathop{\mathrm{End}}\nolimits}}

\def\semi{\mathrel{>\!\!\!\triangleleft}}
\let \DS=\displaystyle

\def \hi{\HC}

\setcounter{secnumdepth}{3} \setcounter{tocdepth}{3}

\def \Fil{\mathop{\mathrm{Fil}}\nolimits}
\def \CoFil{\mathop{\mathrm{CoFil}}\nolimits}
\def \Fill{\mathop{\mathrm{Fill}}\nolimits}
\def \CoFill{\mathop{\mathrm{CoFill}}\nolimits}
\def\SF{{\mathfrak S}}
\def\PF{{\mathfrak P}}
\def \EFil{\mathop{\mathrm{EFil}}\nolimits}
\def \EFill{\mathop{\mathrm{EFill}}\nolimits}
\def \FP{\mathop{\mathrm{FP}}\nolimits}
\def \FPL{\mathop{\mathrm{FPL}}\nolimits}

\let \longto=\longrightarrow
\let \oo=\infty

\let \d=\delta
\let \k=\kappa

\newcommand{\marque}{\addtocounter{smfthm}{1}
{\smallskip \noindent \textit{\thesmfthm}~---~}}

\renewcommand\atop[2]{\ensuremath{\genfrac..{0pt}{1}{#1}{#2}}}

\title[Filtrations de stratification de quelques variÈtÈs de Shimura simples]
{Filtrations de stratification de quelques variÈtÈs de Shimura simples}

\alttitle{Filtration of stratification of some simple Shimura varieties}

\author{Boyer Pascal}
\email{boyer@math.univ-paris13.fr}
\address{ArShiFo ANR-10-BLAN-0114}

\frontmatter

\begin{abstract}
Nous dÈfinissons et Ètudions de nouvelles filtrations dites de stratification d'un faisceau
pervers sur un schÈma; contrairement au cas de la filtration par les poids, ou de monodromie,
ces filtrations sont valables quel que soit l'anneau $\Lambda$ de coefficients. 
Pour $\Lambda=\overline \Qm_l$, nous illustrons
ces constructions dans le contexte des variÈtÈs de Shimura unitaires simples de \cite{h-t} pour
les faisceaux pervers d'Harris-Taylor et le complexe des cycles Èvanescents introduits et 
ÈtudiÈs dans \cite{boyer-invent2}. Nous montrons aussi
comment utiliser ces filtrations afin de simplifier l'Ètape principale de \cite{boyer-invent2}.
Les cas de $\Lambda=\overline \Zm_l$ et $\overline \Fm_l$ seront ÈtudiÈs dans un prochain article.
\end{abstract}

\begin{altabstract}
We define and study new filtrations called of stratification of a perverse sheaf on a scheme; 
beside the cases of the weight or monodromy filtrations, these filtrations are available whatever 
are the ring of coefficients. For $\Lambda=\overline \Qm_l$, we illustrate these constructions in 
the geometric situation of the 
simple unitary Shimura varieties of \cite{h-t} for the perverse sheaves of Harris-Taylor and the
complex of vanishing cycles introduced and studied in \cite{boyer-invent2}. 
We also show how to use these filtrations to simplify
the principal step of \cite{boyer-invent2}. The cases of
$\Lambda=\overline \Zm_l$ and $\overline \Fm_l$ will be studied in another paper. 
\end{altabstract}

\subjclass{14G22, 14G35, 11G09, 11G35,\\ 11R39, 14L05, 11G45, 11Fxx}

\keywords{VariÈtÈs de Shimura, modules formels, correspondances de Langlands, correspondances de
Jacquet-Langlands, faisceaux pervers, cycles Èvanescents, filtration de monodromie, conjecture de
monodromie-poids}

\altkeywords{Shimura varieties, formal modules, Langlands correspondences, Jacquet-Langlands
correspondences,
monodromy filtration, weight-monodromy conjecture, perverse sheaves, vanishing cycles}

\maketitle

\pagestyle{headings} \pagenumbering{arabic}

\section*{Introduction}
\renewcommand{\theequation}{\arabic{equation}}
\backmatter

Pour $l \neq p$ deux nombres premiers distincts,
dans \cite{boyer-invent2} nous avons explicitÈ  le faisceau pervers
des cycles Èvanescents ‡ coefficients dans $\overline \Qm_l$,
d'une certaine classe de variÈtÈs de Shimura unitaires
$X$ qualifiÈe de \og simple \fg{} dans \cite{h-t}, en une place de caractÈristique
rÈsiduelle $p$. Rappelons que cette description
ne repose pas sur une Ètude gÈomÈtrique de ces variÈtÈs qui consisterait ‡ se ramener
‡ une situation semi-stable; pour l'essentiel les arguments se ramËnent 
‡ de la combinatoire sur les reprÈsentations admissibles des groupes linÈaires sur un
corps local et automorphes d'un groupe symplectique sur un corps de nombres.
Au c\oe ur de ces arguments, on trouve la formule des traces de Selberg qui calcule
la somme alternÈe des groupes de cohomologie de certains systËmes locaux dits
d'Harris-Taylor, introduits dans \cite{h-t}, et dÈfinis sur certaines strates de
la fibre spÈciale de ces variÈtÈs de Shimura. L'apport principal de \cite{boyer-invent2}
est l'Ètude des prolongements de ces systËmes locaux ‡ toute la fibre spÈciale. Pour
l'essentiel la preuve consiste
\begin{itemize}
\item ‡ dÈcrire, dans un certain groupe de Grothendieck des faisceaux pervers de Hecke,
les images des extensions par zÈro des systËmes locaux d'Harris-Taylor puis d'en dÈduire 
celle du faisceau pervers $\Psi_\IC$ des cycles Èvanescents.

\item Ensuite on Ètudie la suite spectrale associÈe ‡ la filtration par les poids de 
$\Psi_\IC$, calculant ses faisceaux de cohomologie. 
\end{itemize}
Un fait remarquable est que, pour des raisons combinatoires, 
cette suite spectrale dÈgÈnËre nÈcessairement en $E_2$ et l'Ètape la plus complexe
consiste ‡ montrer que, lorsque ce n'est pas trivialement faux, les $d_1^{p,q}$ sont non
nulles. Pour ce faire, on utilise
\begin{itemize}
\item soit, cf. \cite{boyer-invent2}, une propriÈtÈ d'autodualitÈ ‡ la Zelevinsky sur la 
cohomologie des espaces
de Lubin-Tate qui passe par un thÈorËme difficile de comparaison avec l'espace
de Drinfeld d˚ ‡ Faltings et dÈveloppÈ par Fargues dans \cite{fargues-faltings};

\item soit, cf. \cite{boyer-compositio}, des arguments combinatoires complexes
reposant sur le thÈorËme de Lefschetz difficile, ‡ propos la cohomologie de la variÈtÈ de 
Shimura 
\end{itemize}
Une autre faÁon de voir ces calculs
est de considÈrer la filtration par les noyaux itÈrÈs de la monodromie auquel cas la suite
spectrale associÈe, calculant les faisceaux de cohomologie $\hi^i \Psi_{\IC}$, dÈgÈnËre en
$E_1$. Cette observation nous suggËre une stratÈgie pour montrer que ces
$\hi^i \Psi_\IC$ sont sans torsion: il suffirait 
\begin{itemize}
\item de construire une filtration entiËre de $\Psi_\IC$
qui coÔnciderait sur $\overline \Qm_l$ avec celle par les noyaux itÈrÈs de la monodromie puis

\item de montrer que les faisceaux de cohomologie de ces graduÈs sont sans torsion.
\end{itemize}
Le point de dÈpart de ce travail est donc de construire un telle filtration entiËre, l'Ètude
de la torsion des $\hi^i \Psi_\IC$ Ètant repoussÈe ‡ un prochain article. 

L'idÈe la plus naturelle pour qu'une telle filtration existe quels que soient les coefficients, 
est de lui trouver
une nature gÈomÈtrique et donc d'utiliser des stratifications de l'espace. Ainsi pour un 
faisceau pervers $P$ sur un schÈma muni d'une stratification quelconque, on le filtre
au moyen des morphismes d'adjonction
$j_! j^* P \rightarrow P$ et $P \rightarrow j_* j^* P$, cf. le \S \ref{para-filtration}.
Pour $\Lambda=\Zm_l$, on travaille dans la catÈgorie quasi-abÈlienne des faisceaux pervers
\og libres \fg, cf. le \S \ref{para-quasi} au sens o˘ Ètant donnÈ un faisceau pervers \og libre \fg,
on lui associe une filtration dont les graduÈs sont \og libres \fg. Pour ce faire on
est amenÈ ‡ prendre les coimages, et non les images, des morphismes d'adjonction.
Nous verrons dans un prochain papier que ce phÈnomËne qu'il est naturel de nommer 
\og saturation \fg, est directement liÈ ‡ la torsion. 
On Ètudie, \S \ref{para-def-filtration}, plus particuliËrement deux de ces constructions,
la premiËre dite filtration exhaustive de stratification et la deuxiËme qui lui est duale, la 
cofiltration exhaustive de stratification. 

Au \S \ref{para-shimura0}, on explicite ces constructions pour les extensions par zÈro
des $\overline \Qm_l$-systËmes locaux d'Harris-Taylor, \S \ref{para-systeme-locaux}, puis 
\S \ref{para-psi0} pour le $\overline \Qm_l$-faisceau pervers des cycles Èvanescents 
$\Psi_{\IC,\overline \Qm_l}$. En particulier on vÈrifie que 
la filtration de stratification de $\Psi_{\IC,\overline \Qm_l}$ est Ègale ‡ celle par les noyaux 
de la monodromie, alors que la cofiltration coÔncide ‡ celle par les images.

Rappelons que l'un des intÈrÍts de cet article est de fournir l'ingrÈdient thÈorique
pour contrÙler la torsion des faisceaux de cohomologie de $\Psi_\IC$. Nous montrerons
dans un prochain travail que cette torsion est toujours nulle ce qui, 
par le thÈorËme de Berkovich couplÈ ‡ l'analogue du thÈorËme de Serre-Tate, 
montre que la cohomologie des espaces de Lubin-Tate, 
cf. \cite{boyer-invent2}, est sans torsion et non divisible. 
Afin d'illustrer les techniques qui seront dÈveloppÈes dans ce prochain papier,
nous montrons en appendice, comment d'une part obtenir la filtration de
stratification de $\Psi_{\IC}$ ‡ partir seulement des arguments de somme alternÈe de
\cite{boyer-invent2}, puis comment en dÈduire les $\hi^i \Psi_\IC$.

La lisibilitÈ du texte doit beaucoup ‡ la relecture prÈcise et aux suggestions de J.-F. Dat
qui, en particulier, m'a incitÈ ‡ utiliser le langage des catÈgories quasi-abÈliennes;
je l'en remercie vivement. Merci enfin ‡ D. Juteau pour m'avoir expliquÈ son travail
sur les faisceaux pervers ‡ coefficients entiers.

\tableofcontents

\mainmatter

\renewcommand{\theequation}{\arabic{section}.\arabic{subsection}.\arabic{smfthm}}

\section{Faisceaux pervers entiers et thÈories de torsion}
\label{para-fpentiers}

Dans tout le texte, $\Km$ dÈsignera une extension finie de $\Qm_l$ ou $\overline \Qm_l$,
d'anneau des entiers $\Om$ et de corps rÈsiduel $\Fm=\Om/(\varpi)$. La lettre $\Lambda$
dÈsignera une des trois lettres $\Km,\Om$ ou $\Fm$.

\subsection{ThÈories de torsion}
\label{para-torsion}

Une thÈorie de torsion, cf.\cite{juteau} 1.3.1, sur une catÈgorie abÈlienne
$\AC$ est un couple $(\TC,\FC)$ de sous-catÈgories pleines tel que:
\begin{itemize}
\item pour tout objet $T$ dans $\TC$ et $L$ dans $\FC$, on a
$\hom_\AC(T,L)=0;$

\item pour tout objet $A$ de $\AC$, il existe des objets $A_\TC$ et $A_\FC$ de respectivement $\TC$ et $\FC$,
ainsi qu'une suite exacte courte
$0 \rightarrow A_\TC \longrightarrow A \longrightarrow A_\FC \rightarrow 0.$
\end{itemize}

\rem $\TC$ (resp. $\FC$) est stable par quotients et extensions (resp. sous-objets et extensions).
On dira de $(\TC,\FC)$ est une thÈorie de torsion \textit{hÈrÈditaire} (resp.
\textit{co-hÈrÈditaire}) si $\TC$
(resp. $\FC$) est stable par sous-objets (resp. par quotients).

Dans une catÈgorie abÈlienne $\AC$ qui est $\Om$-linÈaire, un objet $A$ de $\AC$ est dit 
de $\varpi$-torsion (resp. $\varpi$-libre, resp. $\varpi$-divisible) si $\varpi^N 1_A$ est nul pour un certain entier $N$ (resp. $\varpi.1_A$
est un monomorphisme, resp. un Èpimorphisme).

\begin{prop} (cf. \cite{juteau} 1.3.5) \phantomsection \label{prop-noetherienne}
Soit $\AC$ une catÈgorie abÈlienne $\Om$-linÈaire; on note $\TC$ (resp. $\FC$, resp. $\QC$)
la sous-catÈgorie pleine des objets de $\varpi$-torsion (resp. $\varpi$-libres, 
resp. $\varpi$-divisibles) de $\AC$. Si $\AC$
est noethÈrienne (resp. artinienne) alors $(\TC,\FC)$ (resp. $(\QC,\TC)$) est une thÈorie 
de torsion hÈrÈditaire (resp. co-hÈrÈditaire) sur $\AC$.
\end{prop}

\rem dans une catÈgorie abÈlienne $\Om$-linÈaire noethÈrienne (resp. artinienne) tout
$A$ admet un plus grand sous-objet de $\varpi$-torsion $A_{tor}$ 
(resp. $\varpi$-divisible $A_{div}$):
$A/A_{tor}$ (resp. $A/A_{div}$) est sans $\varpi$-torsion (resp. de $\varpi$-torsion) et 
$\Km A \simeq \Km (A/A_{tor})$ (resp. $\Km A \simeq \Km A_{div}$).

\medskip

\begin{notas}\phantomsection \label{nota-pT}
Soit $\DC$ une catÈgorie triangulÈe munie d'une $t$-structure 
$(\DC^{\leq 0},\DC^{\geq 0})$ au sens de \cite{ast} dÈfinition 1.3.1; on notera
$\CC$ son c{\oe}ur, $\tau_{\leq n}$ et $\tau_{\geq n}$, les foncteurs de troncation ainsi
que $\hi^n:= \tau_{\leq n} \tau_{\geq n}=\tau_{\geq n} \tau_{\leq n}$.
Pour $T:\DC_1 \longrightarrow \DC_2$ un foncteur triangulÈ, on note $\lexp p T$ pour
$\hi^0 \circ T \circ \epsilon_1$, o˘ $\epsilon_1:\CC_1 \rightarrow \DC_1$ est l'inclusion 
du c{\oe}ur.
\end{notas}

\rem on rappelle qu'un foncteur triangulÈ $T:\DC_1 \longrightarrow \DC_2$ est dit $t$-exact ‡ droite (resp. ‡ gauche),
si $T(\DC_1^{\leq 0}) \subset \DC_2^{\leq 0}$ (resp. $T(\DC_1^{\geq 0}) \subset 
\DC_2^{\geq 0}$) et $t$-exact s'il l'est ‡ droite et ‡ gauche.
Si $(T^*,T_*)$ est une paire de foncteurs triangulÈs adjoints, alors 
$T^*$ est $t$-exact ‡ droite si et seulement
si $T_*$ est $t$-exact ‡ gauche et alors $(\lexp p T^*,\lexp p T_*)$ est une paire de 
foncteurs adjoints de $\CC_1$ et $\CC_2$.


\begin{coro} \phantomsection \label{coro-pp}(cf. \cite{juteau} 1.3.6)
Si $\CC$ est munie d'une thÈorie de torsion $(\TC,\FC)$
$$\begin{array}{l}
\lexp {+} \DC^{\leq 0}:= \{ A \in \DC^{\leq 1}:~\hi^1(A) \in \TC \} \\
\lexp {+} \DC^{\geq 0}:= \{ A \in \DC^{\geq 0}:~\hi^0(A) \in \FC \} \\
\end{array}$$
dÈfinissent une nouvelle $t$-structure sur $\DC$ dont on note $\lexp + \CC$ le c{\oe}ur
lequel est en outre munie d'une thÈorie de torsion $(\FC,\TC[-1])$.
\end{coro}

\rem on retrouve la $t$-structure $p$ ‡ partir de $p+$ via les formules:
$$\begin{array}{l}
\DC^{\leq 0}:= \{ A \in \lexp {+} \DC^{\leq 0}:~\lexp {+} \hi^0(A) \in \FC \} \\
\DC^{\geq 0}:= \{ A \in \lexp {+} \DC^{\geq -1}:~\lexp {+} \hi^{-1}(A) \in \TC \} \\
\end{array}$$
de sorte que $\lexp {++} \CC=\CC[-1]$.
En particulier pour $A \in \CC$ (resp. $A \in \lexp + \CC$), on a
$$A_\TC= \lexp + \hi^{-1} A \hbox{ et } A_\FC=\lexp + \hi^0 A. \qquad (\hbox{resp. }
A_\FC=\hi^0 A \hbox{ et } A_{\TC[-1]}=\hi^1 A).$$

\subsection{Recollement}
\label{para-recollement}

On suppose donnÈes trois catÈgories triangulÈes $\DC$, $\DC_U$ et $\DC_F$ ainsi que des foncteurs triangulÈs $i_*:\DC_F \rightarrow \DC$ et $j^*:\DC \rightarrow \DC_U$,
vÈrifiant les conditions de recollement de \cite{ast} \S 1.4.3
\begin{enumerate}
\item $i_*$ possËde un adjoint ‡ gauche notÈ $i^*$ et un adjoint ‡ droite $i^!$;

\item $j^*$ possËde un adjoint ‡ gauche $j_!$ et un adjoint ‡ droite $j_*$;

\item on a $j^*i_*=0$ et donc par adjonction $i^*j_!=0$ et $i^!j_*=0$; pour $A \in \DC_F$ et $B \in \DC_U$
$$\hom(j_!B,i_*A)=0 \hbox{ et } \hom(i_*,j_*B)=0;$$

\item pour tout $K \in \DC$, il existe $d:i_*i^*K \rightarrow j_!j^*K[1]$ (resp. $d:j_*j^*K \rightarrow i_*i^!K[1]$),
nÈcessairement unique, tel que le triangle suivant est distinguÈ
$$j_!j^*K \rightarrow K \rightarrow i_*i^* K \leadsto^{d} \quad
(\hbox{resp. } i_*i^! K \rightarrow K \rightarrow j_*j^*K \leadsto^d);$$

\item les foncteurs $i_*$, $j_!$ et $j_*$ sont pleinement fidËles: les morphismes d'adjonction
$i^*i_* \rightarrow \Id \rightarrow i^!i_*$ et $j^*j_* \rightarrow \Id \rightarrow j^*j_!$
sont des isomorphismes.
\end{enumerate}


…tant donnÈes des $t$-structures
$(\DC_U^{\leq 0},\DC_U^{\geq 0})$ sur $\DC_U$ et
$(\DC_F^{\leq 0},\DC_F^{\geq 0})$ sur $\DC_F$, on dÈfinit une $t$-structure sur $\DC$ par recollement:
$$\begin{array}{l}
\DC^{\leq 0}:=\{ K \in \DC:~j^* K \in \DC_U^{\leq 0} \hbox{ et } i^* K \in \DC_F^{\leq 0} \} \\
\DC^{\geq 0}:=\{ K \in \DC:~j^* K \in \DC_U^{\geq 0} \hbox{ et } i^! K \in \DC_F^{\geq 0} \}.
\end{array}$$

\begin{prop} \phantomsection (\cite{juteau} proposition 2.30)
Supposons $\CC_F$ et $\CC_U$ munies de thÈorie de torsion $(\TC_F,\FC_F)$ et 
$(\TC_U,\FC_U)$. On dÈfinit alors une thÈorie de torsion sur $\CC$ par:
$$\begin{array}{l}
\TC:=\{ P \in \CC:~\lexp p i^* P \in \TC_F \hbox{ et } j^* P \in \TC_U \} \\
\FC:=\{ P \in \CC:~\lexp p i^! P \in \FC_F \hbox{ et } j^* P \in \FC_U \}
\end{array}$$
\end{prop}

\begin{coro} \phantomsection \label{coro-torsion0} (\cite{juteau} lemme 2.32)
On a les propriÈtÈs suivantes:
$$\begin{array}{ccc}
\lexp p i_*(\TC_F) \subset \TC & \lexp p j_!(\TC_U) \subset \TC & \lexp p j_{!*}(\TC_U) \subset \TC \\
\lexp p i_*(\FC_F) \subset \FC & \lexp p j_*(\FC_U) \subset \FC & \lexp p j_{!*}(\FC_U) \subset \FC
\end{array}$$
\end{coro}

\begin{proof}
Le rÈsultat dÈcoule de la nullitÈ de 
$\lexp p i^* \lexp p j_{!*}$, $\lexp p i^! \lexp p j_{!*}$ ainsi que celle, cf. \cite{ast} proposition 1.4.17 (i), des composÈs $\lexp p j_* \circ \lexp p i_*$, $\lexp p i^* \circ \lexp p j_!$ et 
$\lexp p i^! \circ \lexp p j_*$.
\end{proof}


\rem la $t$-structure $(\lexp + \DC^{\leq 0}, \lexp + \DC^{\geq 0})$ du corollaire
\ref{coro-pp} s'obtient par recollement, i.e.
$$\lexp + \DC^{\leq 0}:=\{ K \in \DC^{\leq 1}:~\hi^1(K) \in \TC \} =\{ K \in \DC:~j^* K \in
\lexp + \DC_U^{\leq 0} \hbox{ et } i^* K \in \lexp + \DC_F^{\leq 0} \}$$
$$\lexp + \DC^{\geq 0}:=\{ K \in \DC^{\geq 0}:~\hi^0(K) \in \FC \} = \{ K \in \DC:~j^* K \in
\lexp + \DC_U^{\geq 0} \hbox{ et } i^! K \in \lexp + \DC_F^{\geq 0} \}.$$
En effet pour $K \in \DC^{\leq 1}$ avec $\hi^1 (K) \in \TC$, du triangle distinguÈ
$\tau_{\leq 0} K \longrightarrow K \longrightarrow \hi^1 (K) \leadsto$, par application
de $j^*$ (resp. $i^*$), on en dÈduit, comme $j^*(\DC^{\leq 0}) \subset \DC_U^{\leq 0}$
(resp. $i^*(\DC^{\leq 0}) \subset \DC_F^{\leq 0}$), que $\hi^1(j^*K)=j^* \hi^1(K)$
(resp. $\hi^1(i^*K)=i^* \hi^1(K)$) est de torsion et donc $j^* K \in \lexp + \DC_U^{\leq 0}$
(resp. $i^* K \in \lexp + \DC_F^{\leq 0}$). RÈciproquement si $j^*K \in \lexp + \DC_U^{\leq 0}$
et $i^* K \in \lexp + \DC_F^{\leq 0}$ alors $K \in \DC^{\leq 1}$ avec $j^* \hi^1(K) \in \TC_U$
et $i^* \hi^1(K) \in \TC_F$, soit $\hi^1(K) \in \TC$. Le cas de $\lexp + \DC^{\geq 0}$ se traite
de mÍme.

\subsection{La catÈgorie quasi-abÈlienne des faisceaux pervers libres}
\label{para-quasi}

Notons tout d'abord qu'il dÈcoule du corollaire \ref{coro-pp} et de la remarque qui le suit que
$$\FC:=\DC^{\leq 0} \cap \lexp {+} \DC^{\geq 0}=\CC \cap \lexp {+} \CC$$ 
est la sous-catÈgorie
pleine des objets libres (resp. divisibles) de $\CC$ (resp. de $\lexp {+} \CC$).
Les catÈgories $\CC$ et $\lexp {+} \CC$ Ètant abÈliennes, elles possËdent des noyaux,
images, conoyaux et coimages; cependant ces notions peuvent Ítre distinctes selon
qu'on considËre $\CC$ ou $\lexp {+} \CC$, i.e. elles ne fournissent pas nÈcessairement
des objets de $\FC$. 

\begin{lemm}
La catÈgorie $\FC$ admet des noyaux et des conoyaux. Plus prÈcisÈment pour 
$L,L' \in \FC$ et $f:L \longrightarrow L'$, le noyau $\ker_\FC f$ (resp. le conoyau 
$\coker_\FC f$) de $f$ dans $\FC$ est Ègal ‡ celui $\ker_\CC f$ 
(resp. $\coker_{\lexp + \CC} f$) de $f$ dans $\CC$ (resp. $\lexp + \CC$).
\end{lemm}

\begin{proof}
Le noyau $\ker_\CC f$ (resp. $\coker_{\lexp + \CC} f$) est un objet de $\FC$ et 
$\ker_\CC \longrightarrow L \longrightarrow L'$ (resp. $L \longrightarrow L' 
\longrightarrow \coker_{\lexp + \CC} f$) est l'application nulle. Si $X \longrightarrow L$
(resp. $L' \longrightarrow X$) est un morphisme de $\FC$ dont le composÈ avec 
$L \longrightarrow L'$ est nul alors, $X$ Ètant un objet de $\CC$ (resp. de $\lexp + \CC$), 
la propriÈtÈ universelle de 
$\ker_\CC f$ (resp. de $\coker_{\lexp + \CC} f$), nous fournit une flËche
$\ker_\CC \longrightarrow L$ (resp. par $L' \longrightarrow \coker_{\lexp + \CC} f$) qui
factorise $X \longrightarrow L$ (resp. $L' \longrightarrow X$).
Ainsi $\ker_\CC f$ 
(resp. $\coker_{\lexp + \CC} f$) est un noyau (resp. conoyau) de $f$ dans $\FC$.
\end{proof}

\rem d'aprËs la remarque suivant \ref{coro-pp},
$\lexp + \hi^0 \coker_\CC f$ est le quotient libre de $\coker_\CC f$. Si $L'
\longrightarrow X$ est un morphisme de $\FC$ dont le composÈ avec $L \longrightarrow L'$
est nul, alors $X$ Ètant un objet de $\CC$, on a une flËche $\coker_\CC f \longrightarrow X$
telle que le composÈ $(\coker_\CC f)_{tor} \hookrightarrow \coker_\CC f \longrightarrow X$
est nul, i.e. $L' \longrightarrow X$ se factorise par $\lexp + \hi^0 \coker_\CC f$. Ce dernier est
donc Ègal ‡ $\coker_\FC f=\coker_{\lexp + \CC} f$. De la mÍme faÁon on a
$\ker_\FC f=\hi^0 \ker_{\lexp + \CC} f$.

\begin{defi}
Pour $L,L' \in \FC$ et $f:L \longrightarrow L'$, on rappelle que l'image $\im_\FC f$
(resp. la coimage $\coim_\FC f$) de $f$ dans $\FC$ est $\ker_\FC (L' \longrightarrow 
\coker_\FC f)$ (resp. $\coker_\FC ( \ker_\FC f \longrightarrow L)$).
\end{defi}

\begin{lemm}
Pour $f$ un morphisme de $\FC$, on a $\coim_\FC f=\coim_\CC f=\im_\CC f$ et
$\im_\FC f=\im_{\lexp + \CC} f=\coim_{\lexp + \CC} f$.
En outre $f$  induit une flËche canonique
$\coim_\FC f \longrightarrow \im_\FC f.$
\end{lemm}

\begin{proof}
On procËde comme dans la preuve du lemme prÈcÈdent en notant que dans $\CC$
on a $\im_\CC f \hookrightarrow L'$ et donc $\im_\CC f$ est un objet de $\FC$.
De mÍme dans $\lexp + \CC$, on a $L \twoheadrightarrow \coim_{\lexp + \CC} f$
et donc $\coim_{\lexp + \CC} f$ est un objet de $\FC$.

En outre $f:L \longrightarrow L'$ se factorise par $\im_\FC f$ et comme le composÈ
$\ker_\FC f \longrightarrow L \longrightarrow \im_\FC f$ est nul, il induit un morphisme
canonique $\coim_\FC f \longrightarrow \im f$.
\end{proof}

\begin{defi} \label{defi-strict} (cf. \cite{quasi-ab} \S 1.1.3)   \phantomsection
La flËche $f$ est dite \emph{stricte} si le morphisme canonique
$\coim_\FC f \longrightarrow \im_\FC f$ est un isomorphisme; on notera alors $f: L \rarrow L'$.
Une suite $0 \rightarrow L_1 \longrightarrow L_2 \longrightarrow L_3 \rightarrow 0$
sera dite \emph{strictement exacte} si $L_1$ est un noyau de $L_2 \longrightarrow L_3$
et $L_3$ un conoyau de $L_1 \longrightarrow L_2$.
\end{defi}

\rem ($f:L \longrightarrow L'$) est un monomorphisme strict et on note $L \harrow L'$
(resp. un Èpimorphisme strict et on note $L \tarrow L'$) si et seulement si son conoyau 
dans $\CC$ (resp. son noyau dans $\lexp {+} \CC$) est un objet de $\FC$. Dans ce cas
$0 \rightarrow L \longrightarrow L' \longrightarrow \coker_\FC f \rightarrow 0$
(resp. $0 \rightarrow \ker_\FC f \longrightarrow L \longrightarrow L' \rightarrow 0$) est
strictement exacte.

\rem la composÈe de deux monomorphismes (resp. Èpimorphismes) stricts est strict; rÈciproquement si $u \circ v$ est un monomorphisme (resp. Èpimorphisme) strict alors
$v$ (resp. $u$) est strict.\footnote{Ces propriÈtÈs sont gÈnÈrales dans les catÈgories quasi-
abÈliennes cf. la proposition 1.1.8 de \cite{quasi-ab}.}

\begin{prop} 
La catÈgorie $\FC=\CC \cap \lexp {+} \CC$ des objets libres (resp. divisibles) de $\CC$
(resp. de $\lexp {+} \CC$) est quasi-abÈlienne au sens de la dÈfinition 1.1.3 de \cite{quasi-ab}.
\end{prop}

\begin{proof}
Nous avons dÈj‡ vu que $\FC$ possÈdait des noyaux et conoyaux.
Soit alors $f:L_1 \tarrow L_2$ un Èpimorphisme strict et
$$\xymatrix{
L_1 \ar[r] & L_2 \\
L'_1 \ar[r] \ar[u] & L'_2 \ar[u]
}$$
un diagramme cartÈsien. Notons $L=\ker_\FC (L_1 \longrightarrow L_2)$ de sorte que
$0 \rightarrow L \longrightarrow L_1 \longrightarrow L_2 \rightarrow 0$ est une suite
exacte courte de $\CC$. Soit alors $\widetilde L_1$ le tirÈ en arriËre dans $\CC$
$$\xymatrix{
0 \ar[r] & L \ar[r] & L_1 \ar[r] & L_2 \ar[r] & 0 \\
0 \ar[r] & L \ar@{=}[u] \ar[r] & \widetilde L_1 \ar@{-->}[u] \ar@{-->}[r] & L'_2 \ar[r] \ar[u] & 0.
}$$
Comme $L$ et $L'_2$ sont libres dans $\CC$ alors $\widetilde L_1$ est un objet de $\FC$
et l'Èpimorphisme strict $\widetilde L_1 \tarrow L'_2$ se factorise par $L'_1$ de sorte que
$L'_1 \twoheadrightarrow L'_2$ est strict d'aprËs la remarque prÈcÈdente.
Dualement pour $L_1 \longrightarrow L_2$ un monomorphisme strict et un diagramme
cocartÈsien
$$\xymatrix{
L'_1 \ar[r] & L'_2 \\
L_1 \ar[r] \ar[u] & L_2, \ar[u]
}$$
le conoyau $L$ dans $\CC$ de $L_1 \harrow L_2$ est libre. On note $\widetilde L_2$ le 
poussÈ en avant dans $\CC$
$$\xymatrix{
0 \ar[r] & L'_1 \ar@{-->}[r] & \widetilde L_2 \ar[r] & L\ar[r] & 0 \\
0 \ar[r] & L_1 \ar[u] \ar[r] & L_2 \ar@{-->}[u] \ar[r] & L \ar[r] \ar@{=}[u] & 0.
}$$
Comme prÈcÈdemment $\widetilde L_2$ est un objet de $\FC$ et le monomorphisme
strict $L'_1 \harrow \widetilde L_2$ se factorise par $L'_2$ de sorte que d'aprËs la remarque
prÈcÈdente, $L'_1 \hookrightarrow L'_2$ est strict.
Ainsi les conditions
(QA) et (QA*) de \cite{quasi-ab} sont vÈrifiÈes et $\FC$ est quasi-abÈlienne.
\end{proof}

\begin{nota} \label{nota-htarrow}  \phantomsection
Un morphisme $f:L \longrightarrow L'$ qui est ‡ la fois un monomorphisme et un 
Èpimorphisme sera notÈ $f: L \htarrow L'$; on dit aussi que $f$ est un bimorphisme.
\end{nota}

\rem la flËche canonique $\coim_\FC f \longrightarrow \im_\FC f$ est un exemple
de bimorphisme.
Avec ces notations la factorisation canonique d'une flËche dans une catÈgorie
quasi-abÈlienne s'Ècrit comme suit.

\begin{prop}  \phantomsection\label{prop-facto-can}
Tout $f:L \longrightarrow L'$ admet une factorisation canonique
$$L \tarrow \coim_\FC f \htarrow \im_\FC f \harrow L'.$$
\end{prop}

\begin{defi}  \phantomsection \label{defi-sature}
Pour $f$ un monomorphisme, $\im_\FC$ est appelÈ
\emph{le saturÈ} de $\coim_\FC f$.
\end{defi}

\begin{lemm} \phantomsection \label{lem-Tpasnul}
Dans le cas o˘ $\CC$ est $\Om$-linÈaire,
un morphisme $f$ de $\FC$ est un bimorphisme si et seulement si $f\otimes_\Om \Km$ est 
un isomorphisme.
Dans ce cas pour tout entier $n$ assez grand, il existe $g:L' \htarrow L$ tel que
$g \circ f=\varpi^n \Id$.
\end{lemm}

\begin{proof}
Soit $f:L \htarrow L'$ un bimorphisme; $f$ est un monomorphisme
de $\CC$ (resp. un Èpimorphisme de $\lexp + \CC$) de sorte que $f \otimes_\Om \Km$ 
est un monomorphisme (resp. un Èpimorphisme). Comme dans une catÈgorie abÈlienne
les notions de bimorphisme et isomorphisme coÔncident, 
on en dÈduit que $f \otimes_\Om \Km$ est un isomorphisme.
RÈciproquement si $f \otimes_\Om \Km$ est un isomorphisme alors $\ker_\FC f \otimes_\Om
\Km$ et $\coker_\FC f \otimes_\Om \Km$ sont nuls et donc $\ker_\FC f$ et 
$\coker_\FC f$ aussi. Or
$\ker_\FC f=\ker_\CC f$ (resp. $\coker_\FC f=\coker_{\lexp + \CC} f$) de sorte que
$f$ est un monomorphisme de $\CC$ (resp. un Èpimorphisme de $\lexp + \CC$).

Au bimorphisme $L \htarrow L'$ est associÈ une suite exacte
dans $\CC$, $0 \rightarrow L \longrightarrow L' \longrightarrow T \rightarrow 0$
o˘ $T$ est un objet de $\varpi^n$ torsion de $\CC$ pour $n$ assez grand. Soit alors
$$\xymatrix{
0 \ar[r] & L \ar[r] \ar@{=}[d] & L_0 \ar@{-->}[r] \ar@{-->}[d] &
T \ar[r] \ar[d]^{\times \varpi^n} & 0 \\
0 \ar[r] & L \ar@{^{(}->}[d]^{\times \varpi^n} \ar[r] & L'
\ar@{^{(}-->}[d] \ar[r] & T \ar[r] \ar@{=}[d] & 0 \\
0 \ar[r] & L \ar@{-->}[r] & L_0 \ar[r] & T \ar[r] & 0
}$$
o˘ $L_0 \simeq L  \oplus T$. La composÈe 
$L' \longrightarrow L_0 \longrightarrow L$ est alors un monomorphisme de $\CC$
dont le conoyau est de torsion, i.e. $L' \htarrow L$.
\end{proof}

\rem la relation $L \htarrow L'$ est d'Èquivalence sur $\FC$.

\begin{defi}  \phantomsection \label{defi-F-filtration}
Pour $L$ un objet de $\FC$,
on dira que 
$$L_1 \subset L_2 \subset \cdots \subset L_e=L$$ 
est une $\FC$-filtration
si pour tout $1 \leq i \leq e-1$, le monomorphisme
$L_i \hookrightarrow L_{i+1}$ est strict. Dualement 
$L=L_{-e} \twoheadrightarrow L_{1-e} \twoheadrightarrow\cdots \twoheadrightarrow L_{-1}$
est une $\FC$-cofiltration si pour tout $1 \leq i \leq e-1$,  
l'Èpimorphisme $L_{-i-1} \twoheadrightarrow L_{-i}$ est strict.
\end{defi}

ConsidÈrons ‡ nouveau une situation de recollement comme dans le paragraphe
prÈcÈdent. Selon la notation \ref{nota-pT}, considÈrons les foncteurs
$\lexp p j_!$, $\lexp {p+} j_!, \lexp p j_*, \lexp {p+} j_*$ ainsi que les foncteurs
extensions intermÈdiaires $\lexp p j_{!*}$ et $\lexp {p+} j_{!*}$. D'aprËs \cite{juteau} 2.42-2.46,
on a les triangles distinguÈs:
$$\begin{array}{l}
\lexp p j_! \rightarrow \lexp {p+} j_! \rightarrow \lexp p i_* \hi^{-1}_{tors} i^* j_* [1] \leadsto \\
\lexp {p+} j_! \rightarrow \lexp p j_{!*} \rightarrow \lexp p i_* \hi_{libre}^{-1} i^* j_*[1] \leadsto \\
\lexp p j_{!*} \rightarrow \lexp {p+} j_{!*} \rightarrow \lexp p i_* \hi^0_{tors} i^* j_* \leadsto \\
\lexp {p+} j_{!*} \rightarrow \lexp p j_* \rightarrow \lexp p i_* \hi_{libre}^0 i^* j_* \leadsto \\
\lexp p j_* \rightarrow \lexp {p+} j_* \rightarrow \lexp p i_* \hi^1_{tors} i^* j_* [-1] \leadsto
\end{array}$$

\begin{lemm}
Soit $L \in \FC_U$ alors $\lexp {p+} j_! L$, $\lexp p j_{!*} L$, $\lexp {p+} j_{!*} L$
et $\lexp p j_* L$ sont des objets de $\FC$.
\end{lemm}

\begin{nota}  \phantomsection \label{nota-htarrowF}
Pour $L,L' \in \FC$, on notera $f:L \htarrow_+ L'$ un bimorphisme
dont le conoyau dans $\CC$ est dans l'image essentielle de $\CC_F$.
\end{nota}

\noindent \textit{Remarques}: a) pour tout $L \in \FC_U$, 
on voit que $\lexp {p+} j_! L \longrightarrow \lexp p j_{!*} L$ (resp.
$\lexp {p+} j_{!*} L \longrightarrow \lexp p j_* L$) est un Èpimorphisme
(resp. un monomorphisme) strict, et que
$$\lexp {p+} j_! L \tarrow \lexp p j_{!*} L \htarrow_+ \lexp {p+} j_{!*} L \harrow \lexp p j_* L$$
est la factorisation canonique du morphisme
$\lexp {p+} j_! L \longrightarrow \lexp p j_* L$ au sens de \ref{prop-facto-can}.

b) La construction de la preuve du lemme \ref{lem-Tpasnul}, fournit 
$\lexp {p+} j_{!*} L \htarrow \lexp p j_{!*} L$ dont le conoyau dans $\CC$
est $j_{!*} (L/l^nL)$.

\begin{lemm}  \phantomsection \label{lem-htarrow}
Soient $L \in \FC_U$, $M \in \FC_X$ et $f:\lexp p j_{!*} L \htarrow_+ M$. 
Alors  il existe $g: M \htarrow_+ \lexp {p+} j_{!*} L$ 
tel que  $g \circ f:\lexp p j_{!*} L \htarrow_+ \lexp {p+} j_{!*} L$ est le morphisme canonique.
\end{lemm}

\begin{proof}
Soit $M \longrightarrow \lexp p j_* j^* M$ le morphisme d'adjonction et notons
$Q$ le conoyau dans $\FC$ de $\lexp {p+} j_{!*} j^* M \harrow \lexp p j_* j^* M$.
De la nullitÈ de $\hom(\lexp p j_{!*} L, Q)$, on en dÈduit, via l'Èpimorphisme
$\lexp p j_{!*} L \twoheadrightarrow M$ que $M \longrightarrow Q$ est nulle
ce qui fournit, en utilisant $j^*M \simeq L$, une flËche $M \longrightarrow \lexp {p+} j_{!*} L$
qui factorise $\lexp p j_{!*} L \htarrow_+ \lexp {p+} j_{!*} L$ et donc
$M \htarrow_+ \lexp {p+} j_{!*} L$.
\end{proof}

\rem dualement si $f:M \htarrow \lexp {p+} j_{!*} L$ est telle que $\coker_\CC f$
est dans l'image essentielle de $\CC_F$ alors $f:M \htarrow_+ \lexp {p+} j_{!*} L$
factorise $\lexp p j_{!*} L \htarrow \lexp {p+} j_{!*} L$. En particulier on notera que
$f:M \htarrow_+ \lexp p j_{!*} L$ (resp. $f:\lexp {p+} j_{!*} L \htarrow_+ M$)
est nÈcessairement un isomorphisme.

\begin{prop} \phantomsection \label{prop-fp-libre2}
On suppose que $j_*$ (resp. $j_!$) est $t$-exact pour $\CC$ (resp. $\lexp {+} \CC$) alors
$j_* (\FC_U) \subset \FC$ (resp. $j_! (\FC_U) \subset \FC$).
\end{prop}

\begin{proof}
Soit $L \in \FC_U$; comme $\lexp p j_*=j_*$,
$j_* L \in \DC^{\leq 0} \cap \lexp {+} \DC^{\geq 0}=\FC$.
Le cas de $j_!$ est dual: comme
$\lexp {+} j_!=j_!$, alors $j_! L \in \DC^{\leq 0} \cap \lexp {+} \DC^{\geq 0}$ d'o˘ le rÈsultat.
\end{proof}

\subsection{\texorpdfstring{$t$}{Lg}-structures perverses}
\label{para-schema}

Dans la suite $S$ dÈsigne le spectre soit d'un corps, soit
d'un anneau de valuation discrËte hensÈlien $A$, ou celui du normalisÈ
$\bar A$ de $A$ dans une clÙture algÈbrique du corps des fractions de $A$.
On rappelle en outre que $\Lambda$ dÈsigne une des trois lettres $\Km,\Om$ ou $\Fm$ introduites plus haut.
Soit alors $X$ un schÈma de type fini sur $S$. On note $\DC:=D_c^b(X,\Lambda)$
la sous-catÈgorie pleine de $D(X,\Lambda)$ formÈe des complexes ‡ cohomologie bornÈe constructible.

\begin{nota} Dans la suite du texte, $\hi^i$ dÈsignera le foncteur cohomologique associÈe ‡ la $t$-structure naturelle
sur la catÈgorie dÈrivÈe considÈrÈe.
\end{nota}

(a) \textbf{Cas o˘ $S$ est le spectre d'un corps}: on considÈrera dans cette situation
la $t$-structure perverse $p$ dÈfinie par:
$$\begin{array}{l}
A \in \lexp p D^{\leq 0}(X,\Lambda)
\Leftrightarrow \forall x \in X,~\hi^k i_x^* A=0,~\forall k >- \dim \overline{\{ x \} } \\
A \in \lexp p D^{\geq 0}(X,\Lambda) \Leftrightarrow \forall x \in X,~\hi^k i_x^! A=0,~\forall k <- \dim \overline{\{ x \} }
\end{array}$$
o˘ $i_x:\spec \kappa(x) \hookrightarrow X$. On note alors $\lexp p \FP(X,\Lambda)$ le c{\oe}ur de cette $t$-structure: c'est une catÈgorie abÈlienne noethÈrienne et $\Lambda$-linÈaire.

\rem un objet $A$ de $D_c^b(X,\Om)$ est dans $\lexp p D^{\leq 0}(X,\Om)$ 
si et seulement si $A \otimes_\Om^\Lm \Fm$ est un objet de $\lexp p D^{\leq 0}(X,\Fm)$.

\begin{nota} Les foncteurs cohomologiques associÈs ‡ la $t$-structure perverse ci-avant seront notÈs $\lexp p \hi^i$.
\end{nota}

Dans le cas o˘ $\Lambda=\Om$,
en tant que catÈgorie abÈlienne $\Om$-linÈaire, on obtient comme prÈcÈdemment une autre $t$-structure $p+$
$$\begin{array}{l}
A \in \lexp {p+} D^{\leq 0}(X,\Lambda) \Leftrightarrow \forall x \in X,
\left \{ \begin{array}{ll} \hi^i i_x^* A=0, & \forall i >- \dim \overline{\{ x \} } +1 \\
\hi^{-\dim \overline{\{ x \} } +1} i_x^* A & \hbox{de torsion} \end{array} \right. \\
A \in \lexp {p+} D^{\geq 0}(X,\Lambda) \Leftrightarrow \forall x \in X,
\left \{ \begin{array}{ll} \hi^i i_x^! A=0, & \forall i <- \dim \overline{\{ x \} } \\
\hi^{-\dim \overline{\{ x \} }} i_x^! A & \hbox{libre} \end{array} \right.
\end{array}$$
dont on notera $\lexp {p+} \FP(X,\Om)$ le c{\oe}ur et $\lexp {p+} \hi^i$ les foncteurs cohomologiques.
C'est une catÈgorie abÈlienne artinienne et $\Om$-linÈaire.

\rem $A$ est un objet de $\lexp {p+} D^{\geq 0}$
si et seulement si $A \otimes_\Om^\Lm \Fm$ est un objet de $\lexp {p} D^{\geq 0}(X,\Fm)$.

\begin{nota}  \phantomsection \label{nota-fpl}
Nous noterons $\FPL(X,\Om):=\lexp {p+} \FP(X,\Om) \cap \lexp p \FP(X,\Om)$ la catÈgorie 
quasi-abÈlienne des faisceaux pervers \og libres \fg.
\end{nota}

\rem un objet $A$ de $\lexp p \FP(X,\Om)$ est libre si et seulement si 
$A \otimes_\Om^\Lm \Om/(\varpi)$ est pervers.

\rem la dualitÈ de Grothendieck Èchange les deux $t$-structures perverses $p$ et $p+$ 
de sorte que si $f$ est un foncteur exact ‡ droite pour la $t$-structure $p$ et commute 
‡ la dualitÈ de Grothendieck alors il prÈserve la catÈgorie quasi-abÈlienne
des faisceaux pervers libres.

\begin{prop} \phantomsection \label{prop-perversites}
Pour $j:U \hookrightarrow X$ affine et $P \in \FPL(U,\Om)$ on a
$$j_* P =\lexp p j_* P =\lexp {p+} j_* P, \quad j_! P=\lexp p j_! P =\lexp {p+} j_! P,$$
et donc $j_* P$ et $j_! P$ sont des objets de $\FPL(X,\Om)$.
\end{prop}

\begin{proof}
Comme $j$ est affine, $j_*$ (resp. $j_!$) est $t$-exact pour $p$ (resp. $p+$)
et le rÈsultat dÈcoule de la proposition \ref{prop-fp-libre2}.
\end{proof}

(b) \textbf{Cas o˘ $S=\spec A$}: on note
\begin{itemize}
\item $f:X \rightarrow S$ le morphisme structural;

\item $s=\spec k$ (resp. $\eta=\spec K$) le point fermÈ (resp. gÈnÈrique) de $S$;

\item $X_s$ (resp. $X_\eta$) la fibre spÈciale (resp. gÈnÈrique) de $X$;

\item $i:X_s \hookrightarrow X$ (resp. $j:X_\eta \hookrightarrow X$) l'inclusion fermÈe (resp. ouverte).
\end{itemize}
Le complexe $f^! \Lambda_S[2](1)$ est dualisant sur $X$ d'aprËs \cite{Del-finitude} Th. finitude \S 4. On considËre alors la
$t$-structure sur $X$ obtenue par recollement de la perversitÈ autoduale $p$ sur la fibre spÈciale $X_s$ de $X$
et de la $t$-structure $(\lexp p D^{\leq -1}(X_\eta,\Lambda),\lexp p D^{\geq -1}(X_\eta,\Lambda))$, notÈe $p[1]$,
o˘ $p$ est la perversitÈ autoduale sur la fibre gÈnÈrique $X_\eta$ de $X$. Autrement dit, en posant,
d'aprËs \cite{sga43} XIV 2.2, pour $x$ un point de $X$ d'image $y$ dans $S$,
$\delta(x)=\deg \tr \kappa(x)/\kappa(y) + \dim \overline{\{ y \} }$,
on a
$$\begin{array}{l}
A \in \lexp p D^{\leq 0}(X,\Lambda) \Leftrightarrow \forall x \in X,~\hi^q(i_x^* A)=0, \forall q > - \delta(x) \\
A \in \lexp p D^{\geq 0}(X,\Lambda) \Leftrightarrow \forall x \in X,~\hi^q(i_x^! A)=0, \forall q < - \delta(x)
\end{array}$$
On dÈfinit de mÍme la $t$-structure $p+$ sur $X$ de sorte que le foncteur dualisant $D_X=R \hom ( -,K_X)$
Èchange $\lexp p D^{\leq 0}(X,\Lambda)$ (resp. $\lexp {p+} D^{\leq 0}(X,\Lambda)$) et
$\lexp {p+} D^{\geq 0}(X,\Lambda)$ (resp. $\lexp p D^{\geq 0}(X,\Lambda)$).

\begin{prop}
Le foncteur $j_*$ (resp. $j_!$) est $t$-exact pour $X_\eta$
muni de la $t$-structure $p[1]$ (resp. $p+[1]$) et $X$ de la $t$-structure $p$ (resp. $p+$) dÈfinie ci-dessus.
\end{prop}

\begin{proof}
Pour $j_*$, cf. par exemple \cite{ill} bas de la page 48; le cas de $j_!$ est dual.
\end{proof}

On note alors $\lexp {p[1]} \FPL(X_\eta,\Om):=\lexp p D^{\leq -1}(X_\eta,\Om) \cap 
\lexp {p+} 
D^{\geq -1}(X_\eta,\Om)$ la catÈgorie quasi-abÈlienne des faisceaux pervers \og libres \fg.

\begin{coro} (cf. la proposition \ref{prop-fp-libre2})
Si $L_\Om \in \lexp {p[1]} \FPL(X_\eta,\Om)$ alors 
$j_! L_\Om$, $j_* L_\Om$, $\lexp p j_{!*} L_\Om$ et $\lexp {p+} j_{!*} L_\Om$ appartiennent ‡
$\FPL(X,\Om)$, i.e. sont aussi \og libres\fg.
\end{coro}

(c) \textbf{Cas o˘ $S=\spec \bar A$}: on note
\begin{itemize}
\item $\bar s=\spec \bar k$ (resp. $\bar \eta=\spec K$) le point fermÈ (resp. gÈnÈrique) de $S$;

\item $X_{\bar s}$ (resp. $X_{\bar \eta}$) la fibre spÈciale (resp. gÈnÈrique) de $X$;

\item $\bar i:X_{\bar s} \hookrightarrow X$ (resp. $\bar j:X_{\bar \eta} \hookrightarrow X$) l'inclusion fermÈe (resp. ouverte).
\end{itemize}

\rem dans cette situation, il n'y a plus de complexe dualisant.

On considËre alors les $t$-structures suivantes:
\begin{itemize}
\item $p(1)$ en recollant
$(\lexp p D^{\leq -1}(X_{\bar \eta},\Lambda),\lexp p D^{\geq -1}(X_{\bar \eta},\Lambda))$ et
$(\lexp p D^{\leq 0}(X_{\bar s},\Lambda),\lexp p D^{\geq 0}(X_{\bar s},\Lambda))$;

\item $p(0)$ en recollant
$(\lexp p D^{\leq 0}(X_{\bar \eta},\Lambda),\lexp p D^{\geq 0}(X_{\bar \eta},\Lambda))$ et
$(\lexp p D^{\leq 0}(X_{\bar s},\Lambda),\lexp p D^{\geq 0}(X_{\bar s},\Lambda))$.
\end{itemize}

\rem on notera $p(1)+$ et $p(0)+$ les $t$-structures obtenues en recollant comme ci-dessus ‡ partir
des versions $p+$ sur $X_{\bar \eta}$ et $X_{\bar s}$.

\begin{prop} \phantomsection \label{prop-cycles0}
\begin{itemize}
\item[(i)] Pour $X_{\bar \eta}$ et $X$ munis des $t$-structures $p[1]$ et $p(1)$
(resp. $p[1]+$ et $p(1)+$), les foncteurs $\bar j_*$ et $\bar j_!$ sont $t$-exacts.
Pour $L_\Om \in \lexp {p[1]} \FPL(X_{\bar \eta},\Om)$, on a
$$\bar j_* L_\Om=\lexp {p(1)} {\bar j}_{!*} L_\Om=\lexp {p(1)+} {\bar j}_{!*} L_\Om \in 
\lexp {p(1)} \FPL(X,\Om)$$
et on a la suite exacte courte dans $\lexp {p(1)} \FPL(X,\Om)$:
$$0 \rightarrow \bar i_* \bar i^* \bar j_* L_\Om \longrightarrow \bar j_! L_\Om \longrightarrow \bar j_* L_\Om \rightarrow 0.$$

\item[(ii)] Pour $X_{\bar \eta}$ et $X$ munis des $t$-structures $p$ et $p(0)$
(resp. $p+$ et $p(0)+$), les foncteurs $\bar j_*$ et $\bar j_!$ sont $t$-exacts.
Pour $L_\Om \in \lexp {p} \FPL(X_{\bar \eta},\Om)$, on a
$$\bar j_! L_\Om=\lexp {p(0)} {\bar j}_{!*} L_\Om=\lexp {p(0)+} {\bar j}_{!*} L_\Om
\in  \lexp {p(0)} \FPL(X,\Om)$$
et on a la suite exacte courte dans $\lexp {p(0)} \FPL(X,\Om)$:
$$0 \rightarrow \bar j_! L_\Om \longrightarrow j_* L_\Om \longrightarrow 
\bar i_*\bar i^* \bar j_* L_\Om  \rightarrow 0.$$
\end{itemize}
\end{prop}

\rem le complexe $\bar i^* \bar j_* L_\Om$ est le complexe des cycles proches notÈs $\Psi_{\eta}(L_\Om)$
que l'on considËre muni de son action du groupe de Galois.

\begin{proof}
(i) Le foncteur $\bar j_*$ (resp. $\bar j_!$) est $t$-exact ‡ gauche (resp. ‡ droite) pour $p$  et $p+$.
Par ailleurs comme d'aprËs
\cite{ill} \S 4, $\bar i^* \bar j_*$ est $t$-exact relativement ‡ $p$ sur $X_{\bar \eta}$ et $p$
sur $X_{\bar s}$, on en dÈduit que $\bar j_*$ est $t$-exact relativement ‡ $p[1]$ sur $X_{\bar \eta}$
et $p(1)$ sur $X$. En ce qui concerne $\bar j_!$, considÈrons le triangle distinguÈ
$\bar j_! \bar j^* K \longrightarrow K \longrightarrow \bar i_* \bar i^* K \leadsto$
pour $K=\bar j_* L_\Om$ avec $L_\Om$ qui est $p[1]$ pervers.
La suite exacte longue de $p(1)$-cohomologie et le fait que $\lexp p \hi^r \bar i^* \bar j_* L_\Om$ est nul pour
tout $r \neq -1$, donnent la nullitÈ des $\lexp {p(1)} \hi^r \bar j_! \bar j^* K$ pour tout $r \neq 0$ ainsi que
la suite exacte courte
$$0 \rightarrow \bar i_* \lexp {p(1)} \hi^{-1} \bar i^* K \longrightarrow \bar j_! \bar j^* K \longrightarrow K \rightarrow 0$$
avec $\lexp {p(1)} \hi^{-1} \bar i^* K=\Psi_{\eta}(L_\Om)$ qui est $p$-pervers.
On en dÈduit alors que $\bar j_!$ est $t$-exact relativement ‡ $p[1]$ et $p(1)$ ainsi que
$\bar j_* L_\Om=\lexp {p(1)} {\bar j}_{!*} L_\Om$.

Soit alors $L_\Om$ qui est $p[1]+$ pervers et soit
$\lexp {p[1]} \hi^0 L_\Om \rightarrow L_\Om \rightarrow \lexp {p[1]} \tau_{\geq 1} L_\Om \leadsto$
le triangle distinguÈ associÈ o˘ $\lexp {p[1]} \tau_{\geq 1} L_\Om [1]$ est $p[1]$ pervers de torsion et
$\lexp {p[1]} \hi^0 L_\Om$ est libre. AprËs application du foncteur exact $\bar j_*$ on obtient 
$\bar j_* \lexp {p[1]} \hi^0 L_\Om \rightarrow \bar j_* L_\Om \rightarrow \bar j_* \lexp {p[1]} \tau_{\geq 1} L_\Om \leadsto$
o˘ $\bar j_* \lexp {p[1]} \hi^0 L_\Om$ et $\bar j_* \lexp {p[1]} \tau_{\geq 1} L_\Om[1]$ sont $p(1)$-pervers.
Par ailleurs comme pour tout $M$ qui est $p[1]$-pervers, $\bar j_* M = \lexp {p(1)} {\bar j_{!*}} M$, il dÈcoule
du corollaire \ref{coro-torsion0} que $\bar j_* \lexp {p[1]} \hi^0 L_\Om$ (resp. $\bar j_* \lexp {p[1]} \tau_{\geq 1} L_\Om[1]$)
est libre (resp. de torsion). On en dÈduit
alors que $\bar j_*$ est $t$-exact ‡ gauche relativement ‡ $p[1]+$ et $p(1)+$ et donc $t$-exact. La $t$-exactitude de $\bar j_!$
relativement ‡ $p[1]+$ et $p(1)+$ s'en dÈduit alors comme prÈcÈdemment.

\rem pour tout $M$ qui est $p[1]+$ pervers, $\bar j_* M=\lexp {p(1)+} {\bar j}_{!*} M$.

(ii) Le raisonnement est strictement identique en notant que relativement aux $t$-structures $p$ et $p(0)$,
$\bar i_* \bar j^*$ est $t$-exact de sorte que pour $K=\bar j_* L_\Om$,
$\bar i_* \bar i^* K$ et $\lexp {p(1)} {\bar i}^* K=\Psi_\eta(L_\Om)$ sont $p$-pervers alors
que $\lexp {p(1)} \hi^{-1} \bar i^* K$ est nul.
\end{proof}

\section{Filtrations de stratification d'un faisceau pervers sans torsion}
\label{para-filtration}

On suppose le schÈma $X$ muni d'une stratification
$\SF =\{ X=X^{\geq 1} \supsetneq X^{\geq 2} \supsetneq \cdots \supsetneq X^{\geq e} \}$
o˘ pour tout $1 \leq h \leq e$, la strate $X^{\geq h}$ est de pure dimension $d_h$
o˘ $=\dim X=d_1>d_2 \cdots > d_e \geq 0$. On notera 
$j^{\geq h}:X^{=h} \hookrightarrow X^{\geq h}$, $i_{h}: X^{\geq h} \hookrightarrow X$
et $j^{=h}:=i_h \circ j^{\geq h}$.
On rappelle, cf. la dÈfinition \ref{nota-fpl}, que $\FPL(X,\Lambda)$ dÈsigne la catÈgorie
quasi-abÈlienne des faisceaux pervers \og libres \fg; en indice dans les notations, on la notera
simplement $\FC$.

\subsection{Cas d'une situation de recollement}
\label{para-filtration-recollement}

Dans ce paragraphe on s'intÈresse au cas o˘ $e=2$; on notera simplement
$j:U:=X^{\geq 1}-X^{\geq 2} \hookrightarrow X$ et $i:F:=X^{\geq 2}\hookrightarrow X$. 
Pour $L \in \FPL(X,\Lambda)$, on considËre le diagramme suivant 
$$\xymatrix{ 
& L \ar[drr]^{\can_{*,L}} \\
\lexp {p+} j_! j^* L \ar[ur]^{\can_{!,L}} 
\ar@{->>}[r]|-{+} & \lexp p j_{!*}j^* L \ar@{^{(}->>}[r]_+ & 
\lexp {p+} j_{!*} j^* L \ar@{^{(}->}[r]|-{+} & \lexp p j_*j^* L
}$$
o˘ la ligne du bas est la factorisation canonique de
$\lexp {p+} j_! j^* L \longrightarrow \lexp {p} j_* j^* L$, cf. la proposition \ref{prop-facto-can},
et les flËches \og en biais \fg{} donnÈes par adjonction. 

\begin{defi} \phantomsection \label{defi-filtration0}
On note 
$P_L:=i_*\lexp p \hi^{-1}_{libre} i^*j_* j^* L =\ker_\FC \Bigl ( \lexp {p+} j_! 
j^* L \twoheadrightarrow \lexp p j_{!*} j^* L \Bigr ).$
Avec les notations du diagramme ci-dessus, on pose
$$\Fil^0_{U,!}(L)=\im_\FC (\can_{!,L}) \quad \hbox{ et } \quad \Fil^{-1}_{U,!}(L)=\im_\FC \Bigl ( (\can_{!,L})_{|P_L} \Bigr ).$$
\end{defi}

\begin{lemm}  \phantomsection \label{lem-filtration-d2}
Avec les notations prÈcÈdentes, 
$\Fil^{-1}_{U,!}(L) \subset \Fil^0_{U,!}(L) \subset L$
est une $\FC$-filtration au sens de \ref{defi-F-filtration}, avec
$L/\Fil^0_{U,!}(L) \simeq i_* \lexp {p+} i^* L$ et 
$$\lexp p j_{!*} j^* L \htarrow_+ \Fil^0_{U,!}(L)/\Fil^{-1}_{U,!}(L).$$
\end{lemm}

\begin{proof}
Le premier isomorphisme est immÈdiat en utilisant la suite exacte longue de cohomologie
associÈe aux foncteurs $\lexp {p+} \hi^i$ sur 
$j_!j^* F \longrightarrow F \longrightarrow i_*i^* F \leadsto$.
\\
Comme $\Fil_{U,!}^{-1}(L)=\im_\FC \Bigl ( \can_{!,L}:P_L \longrightarrow \Fil^0_{U,!}(L) \Bigr )$,
par dÈfinition de l'image, le monomorphisme $\Fil^{-1}_{U,!}(L) \hookrightarrow \Fil^0_{U,!}(L)$
est strict, i.e. $\gr^0_{U,!}(L):=\Fil^0_{U,!}(L)/\Fil^{-1}_{U,!}(L)$ est un objet de $\FPL(X,\Lambda)$.
Ainsi l'Èpimorphisme $\lexp {p+} j_! j^* L \twoheadrightarrow \Fil^0_{U,!}(L) \tarrow
\gr^0_{U,!}(L)$ se factorise
en un Èpimorphisme $f:\lexp p j_{!*} j^* L \twoheadrightarrow \gr^0_{U,!}(L)$. Or ce dernier
Ètant un isomorphisme sur $U$, $\ker_\FC f$ est de la forme $i_* L'$ avec
$L' \in \FPL(F,\Lambda)$; comme $\hom_{\FC}(i_* L', \lexp p j_{!*} j^* L)$ est nul
on en dÈduit que $L'$ est nul et donc que $f$ est un monomorphisme de $\FPL(X,\Lambda)$,
d'o˘ le rÈsultat.
\end{proof}

\rem d'aprËs le lemme \ref{lem-htarrow},
il existe une flËche $ \Fil^0_{U,!}(L)/\Fil^{-1}_{U,!}(L) \htarrow_+ \lexp {p+} j_{!*} j^* L$
dont la composÈe avec celle de l'ÈnoncÈ est la flËche canonique
$\lexp p j_{!*} j^* L \htarrow_+ \lexp {p+} j_{!*} j^* L$.
$$\xymatrix{
\lexp p j_{!*} j^* L \ar@{^{(}->>}[rr]_+ \ar@{^{(}->>}[drr]_+
& & \Fil^0_{U,!}(L)/\Fil^{-1}_{U,!}(L) \ar@{^{(}->>}[d]_+ \\
& & \lexp {p+} j_{!*} j^* L
}$$
On rÈsume les constructions prÈcÈdentes par le diagramme suivant:
$$\xymatrix{
& \lexp p j_{!*} j^* L \ar@{^{(}->>}[r]_+ & \gr^0_{U,!}(L) \\
& \lexp {p+} j_! j^* L \ar@{->>}[u]|-{+} \ar@{->>}[r] & \Fil^0_{U,!}(L) \ar@{->>}[u]|-{+}
\ar@{^{(}->}[r]|-{+} & L  \ar@{->>}[rr]|-{+} && i_* \lexp {p+} i^* L \\
\lexp p \hi^{-1}_{libre} i^* j_* j^* L \ar@{=}[r] & P_L \ar@{^{(}->}[u]|-{+} \ar@{->>}[r] & \Fil^{-1}_{U,!}(L) 
\ar@{^{(}->}[u]|-{+}
}$$

\rem en utilisant le morphisme $\can_{*,L}$ du diagramme du dÈbut de ce paragraphe,
le lecteur vÈrifiera aisÈment que 
$\Fil^{-1}_{U,!}(L)=\ker_\FC \Bigl ( (\can_{*,L})_{|\im_\FC(\can_{!,L})} 
\Bigr )$ et $\gr^0_{U,!}(L)=\coim_\FC \Bigl ( (\can_{*,L})_{|\im_\FC (\can_{!,L})} \Bigr )$ lequel
est bien \og coincÈ \fg{} entre
$$\lexp p j_{!*} j^* L=\coim_\FC \Bigl ( (\can_{*,L})_{|\coim_\FC(\can_{!,L})} \Bigr ) \hbox{ et }
\lexp {p+} j_{!*} j^* L=\im_\FC \Bigl ( (\can_{*,L})_{|\im_\FC(\can_{!,L})} \Bigr ).$$

\begin{defi} \phantomsection \label{defi-saturee}
La filtration $\Fil_{U,!}^{-1}(L) \subset \Fil^0_{U,!}(L) \subset L$
est dite saturÈe si $\can_{!,L}$ est strict i.e. si
$\Fil^0_{U,!}(L)=\coim_\FC(\can_{!,L}).$
\end{defi}

\rem si $\Fil^0_{U,!}(L)=\coim_\FC(\can_{!,L})$ alors la composÈe
$P_L \harrow j_! j^* L \tarrow \Fil^0_{U,!}(L)$ est stricte i.e.
$\Fil^{-1}_{U,!}(L)=\coim_\FC \Bigl ( (\can_{!,L})_{|P_L}
\Bigr )$. De mÍme l'Èpimorphisme $\lexp p j_{!*} j^* L \tarrow \Fil^0_{U,!}(L) \tarrow \gr^0_{U,!}(L)$
est strict de sorte que ce dernier est un isomorphisme.

\rem la dualitÈ $D$ de Grothendieck-Verdier stabilise $\FPL(X,\Lambda)$, 
Èchange les monomorphismes
(resp. stricts) avec les Èpimorphismes (resp. stricts). En utilisant $D(\im_\FC f)=\coim_\FC
(Df)$ et $D(\can_{!,L})=\can_{*,D(L)}$, on peut dualiser les constructions prÈcÈdentes, ce
qui donne des cofiltrations dÈfinies comme suit.

\begin{defi} \phantomsection \label{defi-cofiltration0}
On note 
$Q_L=i_* \lexp p \hi^0_{libre} i^*j_*j^* L=\coker_\FC \Bigl ( \lexp {p+} j_{!*} j^* L†
\longrightarrow \lexp p j_* j^* L \Bigr )$
et $p_L$ l'Èpimorphisme $p_L:\lexp {p+} j_{!*} j^* L \twoheadrightarrow Q_L$.
On dÈfinit alors
$$\CoFil_{U,*,0}(L)=\coim_\FC(\can_{*,L}) \quad \hbox{ et } \quad 
\CoFil_{U,*,1}(L)=\coim_\FC(p_L \circ \can_{*,L}).$$
\end{defi}

On obtient ainsi une $\FC$-cofiltration,
$L \twoheadrightarrow \CoFil_{U,*,0}(L) \twoheadrightarrow \CoFil_{U,*,1}(L)$
dont le noyau $\cogr_{U,*,1}(L)$ de $\CoFil_{U,*,0}(L) \twoheadrightarrow \CoFil_{U,*,1}(L)$
s'inscrit dans un diagramme commutatif
$$\xymatrix{
\lexp p j_{!*} j^* L \ar@{^{(}->>}[rr]_+ \ar@{^{(}->>}[drr]_+ 
& & \cogr_{U,*,1}(L) \ar@{^{(}->>}[d]_+ \\
& & \lexp {p+} j_{!*} j^* L.
}$$
En ce qui concerne le noyau de $L \twoheadrightarrow \CoFil_{U,*,0}(L)$, en vertu de la suite exacte dans $\CC$:
$$0 \rightarrow i_* \lexp p i^! L \longrightarrow L \longrightarrow \lexp p j_*j^* L 
\longrightarrow i_* \lexp p \hi^1 i^! L \rightarrow 0,$$
on en dÈduit qu'il est isomorphe ‡ $i_* \lexp p i^! L$.
On rÈsume cette construction par le diagramme
$$\xymatrix{
&&&& \cogr_{U,*,1}(L) \ar@{^{(}->>}[r]_+ \ar@{^{(}->}[d]|-{+} & \lexp {p+} j_{!*} j^* L 
\ar@{^{(}->}[d]|-{+} \\
i_* \lexp p i^! L \ar@{^{(}->}[rr]|-{+} &&
L \ar@{->>}[rr]|-{+} && \CoFil_{U,*,0}(L) \ar@{^{(}->}[r] \ar@{->>}[d]|-{+} & \lexp p j_* j^* L
\ar@{->>}[d] \\
&&&& \CoFil_{U,*,1}(L) \ar@{^{(}->}[r] & Q_L \ar@{=}[r] & \lexp p \hi^0_{libre} i^*j_*j^* L
}$$
avec $\CoFil_{U,*,0}(L) \simeq D \Bigl ( \Fil^0_{U,!}(D(L)) \Bigr )$ et
$\CoFil_{U,*,1}(L) \simeq D \Bigl ( \Fil^{-1}_{U,!}(D(L)) \Bigr )$.

\begin{defi} 
La cofiltration $L \twoheadrightarrow \CoFil_{U,*,0}(L) \twoheadrightarrow \CoFil_{U,*,1}(L)$
est dite saturÈe si $\can_{*,L}$ est strict, i.e. si les coimages de la dÈfinition
\ref{defi-cofiltration0} sont Ègales aux images.
\end{defi}

\rem le lecteur notera que $p_L$ Ètant strict si $\can_{*,L}$ est strict alors
$p_l \circ \can_{*,L}$ aussi. L'auteur prÈfÈrant les filtrations aux cofiltrations,
on utilisera la dÈfinition suivante.

\begin{defi}
Pour $\delta=0,1$, on note $\Fil^\delta_{U,*}(L)=\ker_\FC \bigl ( L \twoheadrightarrow
\CoFil_{U,*,\delta}(L) \bigr )$, i.e.
$$\Fil_{U,*}^{0}(L)=\ker_\FC (\can_{*,L}) \quad \hbox{et} \quad \Fil_{U,*}^{1}(L)=\ker_\FC
(p_L \circ \can_{*,L}).$$
On obtient ainsi une $\FC$-filtration
$\Fil_{U,*}^{0}(L) \hookrightarrow \Fil_{U,*}^{1}(L) \hookrightarrow L$.
\end{defi}

\rem de la mÍme faÁon la filtration $\Fil^{-1}_U(L) \subset \Fil^0_U(L) \subset L$ dÈfinit une 
$\FC$-cofiltration
$$L \twoheadrightarrow \CoFil_{U,!,-1}(L) \twoheadrightarrow \CoFil_{U,!,0}(L)$$
avec $\CoFil_{U,!,-1}(L):= L/\Fil^{-1}_{U,!}(L)$ et $\Fil_{U,!,0}(L):=L / \Fil^0_{U,!}(L)$.
La dualitÈ de Verdier-Grothendieck Èchange ces notions, i.e. pour $\delta=0,1$ on a
$\Fil^{\delta}_{U,*}(L) \simeq D \Bigl ( \CoFil_{U,!,-\delta}( D(L)) \Bigr )$ et 
$\Fil^{-\delta}_{U,!}(L)  \simeq D \Bigl ( \CoFil_{U,*,\delta}( D(L)) \Bigr )$.

\subsection{Filtration et cofiltration de stratification}
\label{para-def-filtration0}

Traitons ‡ prÈsent le cas gÈnÈral o˘ la stratification $\SF$ admet Èventuellement 
plus de deux strates, i.e. $e \geq 2$. Pour tout $1 \leq h <e$, on notera
$X^{1 \leq h}:=X^{\geq 1}-X^{\geq h+1}$ et $j^{1 \leq h}:X^{1†\leq h} \hookrightarrow
X^{\geq 1}$ l'inclusion correspondante.

\begin{defi} Pour $L$ un objet de $\FC$ et $1 \leq r \leq e-1$, soit
$$\Fil^r_{\SF,!}(L):=\im_\FC \Bigl ( \lexp {p+} j^{1 \leq r}_! j^{1 \leq r,*} L \longrightarrow L \Bigr ).$$
\end{defi}

\begin{prop} \phantomsection \label{prop-ss-filtration}
La dÈfinition prÈcÈdente munit fonctoriellement tout objet $L$ de $\FC$ d'une 
$\FC$-filtration, au sens de \ref{defi-F-filtration}, dite de $\SF$-stratification
$$0=\Fil^{0}_{\SF,!}(L) \subset \Fil^1_{\SF,!}(L) \subset \Fil^1_{\SF,!}(L) \cdots \subset \Fil^{e-1}_{\SF,!}(L) 
\subset \Fil^e_{\SF,!}(L)=L.$$
\end{prop}

\begin{proof}
Par construction les $\Fil^r_{\SF,!}(L)$ et $L/\Fil^r_{\SF,!}(L)$ sont des objets de $\FC$.
En outre comme $j^{1 \leq r-1,*} \Fil^r_{\SF,!}(L) \simeq j^{1 \leq r-1,*} L$, le morphisme
d'adjonction $\lexp {p+} j^{1 \leq r-1}_! j^{1 \leq r-1,*} L \longrightarrow L$ se factorise par 
$\Fil^r_{\SF,!}(L)$ 
$$\xymatrix{
&  \Fil^{r-1}_{\SF,!}(L) \ar[dl] \ar@{^{(}->}[d]|-{+} \\
\Fil^r_{\SF,!}(L) \ar@{^{(}->}[r] & L
}$$
de sorte que $\Fil^{r-1}_{\SF,!}(L) \longrightarrow \Fil^r_{\SF,!}(L)$ est un monomorphisme strict.
\end{proof}

\begin{defi} \phantomsection \label{defi-SF-saturee}
On dira que $L$ est $\SF_!$-saturÈ si pour tout  $1 \leq r \leq e-1$, le morphisme
d'adjonction $\lexp {p+} j^{1 \leq r}_! j^{1 \leq r,*} L \longrightarrow L$ est strict, i.e. si
$\Fil^r_{\SF,!}(L)=\coim_\FC \Bigl ( \lexp {p+}j^{1 \leq r}_! j^{1 \leq r,*} L \longrightarrow L \Bigr )$.
Autrement dit si pour tout $1 \leq r \leq e-1$, $\lexp p i_{r+1}^* L$ est un objet de $\FC$.
\end{defi}

\rem $\Fil^{r-1}_{\SF,!}(L)$ Ètant construit, on peut aussi obtenir $\Fil^r_{\SF,!}(L)$ comme suit.
Le quotient $L/\Fil^{r-1}_{\SF,!}(L)$ est ‡ support dans $X^{\geq r}$ et on considËre
$i_{r,*} \Fil^0_{X^{=r},!} \Bigl ( L/\Fil^{r-1}_{\SF,!}(L) \Bigr )$ au sens de la dÈfinition \ref{defi-filtration0} 
pour $U=X^{=r} \hookrightarrow X^{\geq r}$. Soit alors $\widetilde \Fil^r_{\SF,!}(L)$ 
le tirÈ en arriËre 
$$\xymatrix{
\Fil^{r-1}_{\SF,!}(L) \ar@{^{(}->}[r] & L \ar@{->>}[r] & L/\Fil^{r-1}_{\SF,!}(L) \\
\Fil^{r-1}_{\SF,!}(L) \ar@{^{(}->}[r] \ar@{=}[u] &
\widetilde \Fil^r_{\SF,!}(L) \ar@{^{(}-->}[u] \ar@{-->>}[r] & i_{r,*}\Fil^0_{X^{=r},!} \Bigl ( L/\Fil^{r-1}_{\SF,!}(L) \Bigr ) 
\ar@{^{(}->}[u]
}$$
Par construction $\lexp {p+} \hi^0 i_r^* \Bigl ( \Fil^{r-1}_{\SF,!}(L) \Bigr )$ est nul et donc
$\lexp {p+} \hi^0 i_{r+1}^* \Bigl ( \Fil^{r-1}_{\SF,!}(L) \Bigr )$ aussi. Ainsi la nullitÈ
de $\lexp {p+} \hi^0 i_{r+1}^* \Bigl ( i_{r,*}
\Fil^0_{X^{=r},!} \bigl ( L/\Fil^{r-1}_{\SF,!}(L) \bigr ) \Bigr )$
implique celle de $\lexp {p+} \hi^0 i_{r+1}^* \widetilde \Fil^r_{\SF,!}(L) $ de sorte que
$\widetilde \Fil^r_{\SF,!}(L)$ est la coimage du morphisme d'adjonction
$\lexp {p+} j^{\geq r}_! j^{\geq r,*} \Bigl ( \widetilde \Fil^r_{\SF,!}(L)  \Bigr ) \longrightarrow 
\widetilde \Fil^r_{\SF,!}(L)$. En prenant le $j^{1 \leq r,*}$ du diagramme prÈcÈdent,
on en dÈduit que $j^{1 \leq r,*} L \simeq j^{1 \leq r,*} \widetilde \Fil^r_{\SF,!}(L) $ de sorte que
$\widetilde \Fil^r_{\SF,!}(L)  \simeq \Fil^r_{\SF,!}(L)$.

\begin{defi}
Pour $L$ un objet de $\FC$ et $1 \leq r \leq e-1$, soit
$$\CoFil_{\SF,*,r}(L)=\coim_\FC \Bigl ( L \longrightarrow \lexp {p} j^{1 \leq r}_* j^{1 \leq r,*} L \Bigr ).$$
\end{defi}

\begin{prop} La dÈfinition prÈcÈdente munit fonctoriellement tout objet $L$ de $\FC$
d'une $\FC$-cofiltration dite de $\SF$-stratification
$$L = \CoFil_{\SF,*,e}(L) \twoheadrightarrow \CoFil_{\SF,*,e-1}(L) \twoheadrightarrow \cdots 
\twoheadrightarrow \CoFil_{\SF,*,1}(L) \twoheadrightarrow \CoFil_{\SF,*,0}(L)=0.$$
\end{prop}

\rem comme prÈcÈdemment, on peut construire
cette cofiltration de maniËre itÈrative. Supposons construit $\CoFil_{\SF,*,r-1}(L)$
et notons $K_r=\ker_\FC \Bigl ( L \tarrow \CoFil_{\SF,*,r-1}(L) \Bigr )$ ‡ support dans
$X^{\geq r}$. Soient alors $K_r \tarrow i_{r,*}\CoFil_{X^{=r},*,0} (K_r)$ et le poussÈ en avant
$$\xymatrix{
K_r \ar@{->>}[d] \ar@{^{(}->}[r] & L \ar@{-->>}[d] \ar@{->>}[r] & \CoFil_{\SF,*,r-1}(L) \ar@{=}[d] \\
i_{r,*} \CoFil_{X^{=r},*,0} (K_r) \ar@{^{(}-->}[r] & \widetilde \CoFil_{\SF,*,r}(L) \ar@{->>}[r] &  
\CoFil_{\SF,*,r-1}(L).
}$$
Comme prÈcÈdemment on montre que 
$\CoFil_{\SF,*,r}(L) \simeq \widetilde \CoFil_{\SF,*,r}(L) $.

\begin{defi}
Pour tout $0 \leq r \leq e$, on note 
$$\Fil_{\SF,*}^{-r}(L):=\ker_\FC \bigl ( L \twoheadrightarrow \CoFil_{\SF,*,r}(L) \bigr ),$$
i.e. $\Fil_{\SF,*}^{-r}(L)=\ker_\FC \bigl ( \can_{*,L}:L \longrightarrow
 \lexp {p} j^{1 \leq r}_* j^{1 \leq r,*} L \bigr )$. On obtient ainsi une $\FC$-filtration
 $$0=\Fil_{\SF,*}^{-e}(L) \subset \Fil_{\SF,*}^{1-e}(L) \subset \cdots \subset \Fil_{\SF,*}^0(L)=L$$
 dont on note $\gr_{\SF,*}^{-r}(L)$ les graduÈs. On dit que $L$ est $\SF_*$-saturÈ si
 pour tout $1 \leq r \leq e-1$, l'Èpimorphisme $L \twoheadrightarrow \CoFil_{\SF,*,r}(L)$
 est strict.
\end{defi}

\rem on peut aussi dÈfinir la cofiltration $\CoFil_{\SF,!,-r}(L):=L/\Fil_{\SF,!}^r(L)$.
Comme prÈcÈdemment la dualitÈ de Verdier Èchange ces notions, i.e.
$D \Bigl ( \CoFil_{\SF,!,-r}(L) \Bigr ) \simeq \Fil^{-r}_{\SF,*}( D(L))$, et
$D \Bigl ( \CoFil_{\SF,*,r}(L) \Bigr ) \simeq \Fil^{r}_{\SF,!}( D(L))$.

\subsection{Filtrations exhaustives de stratification}
\label{para-def-filtration}

Dans ce paragraphe on cherche ‡ raffiner les filtrations du paragraphe prÈcÈdent 
de faÁon ‡ obtenir la filtration la plus fine possible \og adaptÈe \fg{} ‡ la stratification
$\SF$ au sens suivant.

\begin{defi} \label{defi-adpate}
Un faisceau pervers $L \in \FPL(X,\Lambda)$ est dit \emph{adaptÈ ‡ la stratification} $\SF$
s'il existe $1 \leq h \leq e$ tel que $L \simeq i_{h,*} i_h^* L$ et que le morphisme
d'adjonction $j^{\geq h}_! j^{\geq h,*} \bigl ( i_h^* L \bigr ) \longrightarrow i_h^* L$
induit un bimorphisme $\lexp p j^{= h}_{!*} j^{= h,*} L \htarrow_+ L$ o˘ l'on rappelle que
$i_{h,*}=i_{h,!}=\lexp p i_{h,!*}=\lexp {p+} i_{h,!*}$ et $j^{=h}:=i_h \circ j^{\geq h}$.
Une $\FC$-filtration sera dite \emph{adaptÈe} ‡ $\SF$ si tous ses graduÈs le sont.
\end{defi}

\rem les filtrations $\Fil_{\SF,!}^\bullet$ et $\Fil_{\SF,*}^\bullet$ du paragraphe
prÈcÈdent ne sont, en gÈnÈral, pas adaptÈes ‡ $\SF$. Une faÁon d'y remÈdier serait de les
combiner en considÈrant les $\gr^{i_1}_{\SF,*} \gr^{j_1}_{\SF,!} \cdots \gr^{i_e}_{\SF,*}
\gr^{j_e}_{\SF,!} (L)$. Une autre, plus simple, consiste ‡ utiliser
les $\Fil^{-1}_{U,!}$ du \S \ref{para-filtration-recollement}.

\begin{nota}
Pour $e \geq 2$ et $r \in \Zm-\{ 0 \}$, on note $v_{2,e}(r)$ la valuation $2$-adique de $r$
i.e. $r=2^{v_{2,e}(r)}m$ avec $m$ impair et on pose $v_{2,e}(0)=e-1$.
\end{nota}

\begin{prop}
\'Etant donnÈe une stratification 
$\SF =\{ X=X^{\geq 1} \supsetneq \cdots \supsetneq X^{\geq e} \}$
de $X$, tout objet $L$ de $\FPL(X,\Lambda)$ est muni fonctoriellement 
d'une $\FC$-filtration dite exhaustive de $\SF$-stratification adaptÈe ‡ $\SF$ au sens
de la dÈfinition prÈcÈdente
$$0=\Fill^{-2^{e-1}}_{\SF,!}(L) \subset \Fill^{-2^{e-1}+1}_{\SF,!}(L) \subset \cdots \subset
\Fill^0_{\SF,!}(L) \subset \cdots \subset \Fill^{2^{e-1}-1}_{\SF,!}(L)=L,$$
telle que pour tout $|r|<2^{e-1}$, il existe $P_r \in \FPL(X^{=e-v_2(r)},\Lambda)$ tel que 
$\grr^r_{\SF,!}(L):=\Fill^r_{\SF,!}(L)/\Fill^{r-1}_{\SF,!}(L)$ s'inscrive dans un diagramme
commutatif de bimorphismes
$$\xymatrix{
\lexp p j^{= e-v_2(r)}_{!*} P_r \ar@{^{(}->>}[rr]_+ \ar@{^{(}->>}[dr]_+ & & 
\lexp {p+} j^{= e-v_2(r)}_{!*} P_r \\
& \grr^r_{\SF,!}(L) \ar@{^{(}->>}[ur]_+
}$$
\end{prop}

\begin{proof}
On va montrer le rÈsultat par rÈcurrence sur $e$, le cas $e=2$ Ètant donnÈ 
par la construction $\Fil^{-1}_{U,!}(L) \harrow \Fil^0_{U,!}(L) \harrow \Fil^1_{U,!}(L):=L$ du 
\S \ref{para-filtration-recollement}. Supposons donc le rÈsultat acquis jusqu'au rang $e-1$
et traitons le cas de $e$.
Partant d'un objet $L$ de $\FPL(X,\Lambda)$, on note
$\Fill^0_{\SF,!}(L):=\Fil^1_{\SF,!}(L)=\Fil^0_{X^{=1}}(L)$ et
$\Fill^{-1}_{\SF,!}(L)=\Fil^{-1}_{X^{=1}}(L)$ o˘ le graduÈ 
$\grr^0_{\SF,!}(L):=\Fill^0_{\SF,!}(L)/\Fill^{-1}_{\SF,!}(L)$ \og factorise \fg{} le morphisme canonique
$\lexp p j^{\geq 1}_{!*} \bigl ( j^{\geq 1,*} L \bigr ) \htarrow_+ \lexp {p+} 
j^{\geq 1}_{!*} \bigl ( j^{\geq 1,*} L \bigr )$, i.e.
$$\xymatrix{
\lexp p j^{\geq 1}_{!*} \bigl ( j^{\geq 1,*} L \bigr ) \ar@{^{(}->>}[rr]_+ \ar@{^{(}->>}[dr]_+ & & 
\lexp {p+} j^{\geq 1}_{!*} \bigl ( j^{\geq 1,*} L \bigr ) \\
& \grr^0_{\SF,!}(L). \ar@{^{(}->>}[ur]_+
}$$
Comme $\Fill^{-1}_{\SF,!}(L)$ (resp. $L/\Fill^0_{\SF,!}(L)$) est ‡ support dans $X^{\geq 2}$, 
muni de la stratification $\SF^1:=\{ X^{\geq 2} \supsetneq \cdots \supsetneq X^{\geq e} \}$,
il est, d'aprËs l'hypothËse de rÈcurrence, muni d'une filtration
$$0=\Fill^{-2^{e-2}}_{\SF^1,!} \Bigl ( \Fill^{-1}_{\SF,!}(L) \Bigr ) \subset \cdots \subset
\Fill^0_{\SF^1,!} \Bigl ( \Fill^{-1}_{\SF,!}(L) \Bigr ) \subset \cdots  \subset 
\Fill^{2^{e-2}-1}_{\SF^1,!} \Bigl ( \Fill^{-1}_{\SF,!}(L) \Bigr )=\Fill^{-1}_{\SF,!}(L),$$
$$\bigl (\hbox{resp. }
0=\Fill^{-2^{e-2}}_{\SF^1,!} \Bigl ( L/ \Fill^0_{\SF,!}(L) \Bigr ) \subset \cdots \subset 
\Fill^{2^{e-2}-1}_{\SF^1,!} \Bigl ( L/ \Fill^0_{\SF,!}(L) \Bigr )=L/ \Fill^0_{\SF,!}(L). \bigr )$$
Pour tout $|r|< 2^{e-2}$, on pose
$\Fill_{\SF,!}^{-2^{e-2}+r}(L)=\Fill^r_{\SF^1,!} \Bigl ( \Fill^0_{\SF,!}(L) \Bigr )$
ainsi que le tirÈ en arriËre:
$$\xymatrix{
0 \ar[r] & \Fill^0_{\SF,!}(L) \ar[r] \ar@{=}[d] & L \ar[r] & L/\Fill^0_{\SF,!}(L) \ar[r] & 0 \\
0 \ar[r] & \Fill^0_{\SF,!}(L) \ar[r] & \Fill_{\SF,!}^{2^{e-2}+r}(L) \ar@{^{(}-->}[u]  \ar@{-->}[r] &
\Fill^r_{\SF^1,!} \Bigl ( L/\Fill^0_{\SF,!}(L) \Bigr ) \ar@{^{(}->}[u] \ar[r] & 0.
}$$
En notant que pour $|r|< 2^{e-2}$, $v_{2,e}(r)=v_{2,e}(\pm 2^{e-2}+r)$, on voit que
la propriÈtÈ sur les graduÈs est clairement vÈrifiÈe.
\end{proof}

\begin{defi} \phantomsection \label{defi-satureee}
On dira que $L$ est exhaustivement $\SF_!$-saturÈ si dans la construction prÈcÈdente 
tous les
morphismes d'adjonction considÈrÈs $\can_{!,P_r}$ sont stricts, cf. \ref{defi-saturee}.
\end{defi}

\rem comme prÈcÈdemment, en utilisant les $L' \twoheadrightarrow \CoFil_{U,*,0}(L') 
\twoheadrightarrow \CoFil_{U,*,1}(L')$, on construit une cofiltration exhaustive de
$\SF$-stratification
$$L=\CoFill_{\SF,*,2^{e-1}}(L) \twoheadrightarrow \CoFill_{\SF,*,2^{e-1}-1}(L)
\twoheadrightarrow \cdots \twoheadrightarrow \CoFill_{\SF,*,-2^{e-1}}(L)=0$$
vÈrifiant des propriÈtÈs analogues ‡ celles de la proposition prÈcÈdente.

\begin{defi}
Avec les notations prÈcÈdentes, pour tout $|r|<2^{e-1}$,
$$\Fill_{\SF,*}^{-r}(L)=\ker_\FC \bigl ( L \twoheadrightarrow \CoFill_{\SF,*,r}(L) \bigr )$$
dÈfinit une $\FC$-filtration dont on note $\grr_{\SF,*}^r(L)$ les graduÈs. Cette filtration
sera dite $\SF_*$-saturÈe si les Èpimorphismes $L \twoheadrightarrow \CoFil_{\SF,*,r}(L)$
sont stricts.
\end{defi}

\rem comme prÈcÈdemment la dualitÈ de Verdier Èchange filtrations et cofiltrations, i.e.
$D \bigl ( \CoFill_{\SF,*,r}(L) \bigr ) \simeq \Fill^r_{\SF,!} \bigl ( D(L) \bigr )$
et $D \bigl ( \CoFill_{\SF,!,r}(L) \bigr ) \simeq \Fill^r_{\SF,*} \bigl ( D(L) \bigr )$.
En particulier si $L \simeq D(L)$ est autodual, alors il est $\SF_!$-saturÈ si et seulement
s'il est $\SF_*$-saturÈ.

\rem pour $\Om$ un anneau de coefficients, soit $L_0 \harrow L_1$ un monomorphisme 
strict de $\FPL(X,\Om)$ dont le conoyau est de 
la forme $\bar i_*\lexp p j_{!*} P \htarrow L_1/L_0 \htarrow \bar i_* \lexp {p+} j_{!*} P$
o˘ $j:U \hookrightarrow \bar U$ est une inclusion ouverte et $i:\bar U \hookrightarrow X$
est une immersion fermÈe. On suppose donnÈ $G \in \FPL(X,\Om)$ tel que
$\bar i_*\lexp p j_{!*} P \htarrow G \htarrow \bar i_* \lexp {p+} j_{!*} P$.
Comme d'aprËs le lemme \ref{lem-Tpasnul}, $\htarrow$ est une relation d'Èquivalence
et que $L_1/L_0$ et $G$ sont Èquivalents ‡ $\bar i_*\lexp p j_{!*} P$, il existe une flËche
$G \htarrow L_1/L_0$ de sorte que le tirÈ en arriËre
$$\xymatrix{
L_0 \ar@{^{(}->}[r] \ar@{=}[d] & L'_1 \ar@{-->>}[r] \ar@{^{(}-->>}[d] & G \ar@{^{(}->>}[d] \\
L_0 \ar[r] & L_1 \ar@{->>}[r] & L_1/L_0
}$$
fournit un objet $L'_1 \htarrow L_1$ munit d'une $\FC$-filtration $L_0 \subset L'_1$
dont on a modifiÈ \og la position \fg{} du graduÈ $L'_1/L_0$ par rapport ‡ celle de $L_1/L_0$.
Plus gÈnÈralement le lecteur se convaincra qu'Ètant donnÈs:
\begin{itemize}
\item un objet $L$ de $\FC$ dont la filtration de $\SF$-stratification exhaustive fournit 
pour tout $r \in \Zm$ tel que $|r| < 2^{e-1}$ des faisceaux pervers $P_r$ sur 
$X^{=e-v_2(r)}$ et $\lexp p j^{=e-v_2(r)}_{!*} P_r \htarrow \grr^r_\SF(L) 
\htarrow \lexp {p+} j^{= e-v_2(r)}_{!*} P_r$, 

\item pour tout $|r|<2^{e-1}$ des
$\lexp p j^{= e-v_2(r)}_{!*} P_r \htarrow G_r \htarrow 
\lexp {p+} j^{= e-v_2(r)}_{!*} P_r$,
\end{itemize}
on peut construire $L'$ dont les graduÈs de la filtration exhaustive 
sont les $G_r$.

\section{Faisceaux pervers sur les variÈtÈs de Shimura simples}
\label{para-shimura0}

DÈsormais nous nous placerons dans le cas $\Lambda=\overline \Qm_l$; le but de ce 
paragraphe est d'expliciter les constructions prÈcÈdentes sur le faisceaux pervers des 
cycles Èvanescents des variÈtÈs de Shimura simples de \cite{h-t}.

\subsection{Rappels sur les reprÈsentations de \texorpdfstring{$GL_n(K)$}{Lg}}
\label{para-rappel-rep}

Dans la suite $K$ dÈsigne un corps local non archimÈdien dont le corps rÈsiduel est de 
cardinal $q$ une puissance de $p$ et on rappelle quelques notations de 
\cite{boyer-invent2} sur les reprÈsentations admissibles de $GL_n(K)$ ‡ coefficients dans 
$\overline \Qm_l$ o˘ $l$ un nombre premier distinct de $p$.

\begin{nota}
Une racine carrÈe $q^{\frac{1}{2}}$ de $q$ dans $\overline \Qm_l$ Ètant fixÈe,
pour $k \in \frac{1}{2} \Zm$, nous noterons
$\pi\{ k \}$ la reprÈsentation tordue de $\pi$ de sorte que l'action d'un ÈlÈment $g \in GL_n(K)$ est donnÈe par $\pi(g) \nu(g)^k$ avec
$\nu: g \in GL_n(K) \mapsto q^{-\val (\det g)}$.
\end{nota}

\begin{defi}
Pour $(\pi_1,V_1)$ et $(\pi_2,V_2)$ des $R$-reprÈsentations de respectivement 
$GL_{n_1}(K)$ et $GL_{n_2}(K)$, et $P_{n_1,n_2}$ le parabolique standard de 
$GL_{n_1+n_2}$ de Levi $M=GL_{n_1} \times GL_{n_2}$, on notera
$$\pi_1 \times \pi_2:=\ind_{P(K)}^{GL_{n_1+n_2}(K)} \pi_1\{ n_2/2 \} \otimes \pi_2 \{ -n_1/2\}.$$
\end{defi}

On rappelle qu'une reprÈsentation irrÈductible admissible $\pi$ de $GL_n(K)$ est 
dite \textit{cuspidale} si elle n'est pas un sous-quotient d'une induite parabolique propre.
Pour $g$ un diviseur de $d=sg$ et $\pi$ une reprÈsentation cuspidale
irrÈductible de $GL_g(K)$, l'induite parabolique
$$\pi\{ \frac{1-s}{2} \} \times \pi\{\frac{3-s}{2} \} \times \cdots \times \pi\{ \frac{s-1}{2} \}$$
possËde un unique quotient (une unique sous-reprÈsentation) 
irrÈductible notÈ $\st_s(\pi)$ (resp. $\speh_s(\pi)$).
D'aprËs \cite{zelevinski2} 2.10, l'induite parabolique
$\st_{s-t}(\pi \{ - \frac{t}{2} \} ) \times \speh_{t}(\pi \{ \frac{s-t}{2} \} )$
admet une unique sous-reprÈsentation irrÈductible que l'on note $LT_\pi(s,t)$.
Afin d'Èviter d'avoir ‡ Ècrire systÈmatiquement toutes ces torsions, on introduit la notation suivante.

\begin{nota} Un entier $g \geq 1$ Ètant fixÈ, pour $\pi_1$ et $\pi_2$ des reprÈsentations de respectivement $GL_{t_1g}(K)$ et $GL_{t_2g}(K)$, on notera
$$\pi_1 \overrightarrow{\times} \pi_2 =\pi_1 \{ -\frac{t_2}{2} \} \times \pi_2 \{ \frac{t_1}{2} \}.$$
\end{nota}

\subsection{VariÈtÈs de Shimura unitaires simples}
\label{para-shimura}

Soit $F=F^+ E$ un corps CM avec $E/\Qm$ quadratique imaginaire pure, dont on fixe un plongement rÈel
$\tau:F^+ \hookrightarrow \Rm$. Dans \cite{h-t}, les auteurs justifient l'existence d'un groupe unitaire $G_\tau$
vÈrifiant les points suivants:
\begin{itemize}
\item $G_\tau(\Rm) \simeq U(1,d-1) \times U(0,d)^{r-1}$;

\item $G_\tau(\Qm_p) \simeq (\Qm_p)^\times \times \prod_{i=1}^r (B_{v_i}^{op})^\times$ o˘ $v=v_1,v_2,\cdots,v_r$ sont les
places de $F$ au dessus de la place $u$ de $E$ telle que $p=u \lexp c u$ et o˘ $B$ est une algËbre ‡ division centrale sur $F$
de dimension $d^2$ vÈrifiant certaines propriÈtÈs, cf. \cite{h-t}, dont en particulier d'etre soit dÈcomposÈe soit une algËbre ‡
division en toute place et dÈcomposÈe ‡ la place $v$. 
\end{itemize}

Pour tout sous-groupe compact $U^p$ de $G_\tau(\Am^{\oo,p})$ et $m=(m_1,\cdots,m_r) \in \Zm_{\geq 0}^r$, on pose
$$U^p(m)=U^p \times \Zm_p^\times \times \prod_{i=1}^r \ker ( \OC_{B_{v_i}}^\times \longto
(\OC_{B_{v_i}}/v_i^{m_i})^\times )$$

\begin{defi}
Pour $U^p$ \og assez petit \fg{}\footnote{tel qu'il existe une place $x$ pour laquelle la projection de $U^p$ sur $G(\Qm_x)$ ne contienne
aucun ÈlÈment d'ordre fini autre que l'identitÈ, cf. \cite{h-t} bas de la page 90} soit
$X_{U^p(m)}$ \og la variÈtÈ de Shimura associÈe ‡ $G$\fg{} construite dans \cite{h-t}.
On notera $\IC$ l'ensemble de ces $U^p$.
\end{defi}

On note $\IC$ l'ensemble des sous-groupes compacts ouverts \og assez petits \fg{} 
de $\Gm$, de la forme $U^p(m)$ et donc muni d'une application
$m_1:\IC \longto \Nm$.

Pour tout $I \in \IC$, $X_{I}$ est un schÈma projectif sur $\spec \OC_v$; les
$(X_{I})_{I \in \IC}$ forment un systËme projectif notÈ simplement $X_\IC$ muni
d'une action par correspondance de $G_\tau(\Am^\oo)$,
o˘ les morphismes de restriction du niveau $g_{J,I}$ sont finis et plats et mÍme Ètale si 
$m_1(I)=m_1(J)$.

\begin{defi} (cf. \cite{boyer-invent2} \S 1.3)
Pour tout $1 \leq h \leq d$, $X_{I,\bar s}^{\geq h}$ (resp. $X_{I,\bar s}^{=h}$)
dÈsigne la strate fermÈe (resp. ouverte) de Newton de hauteur $h$ de la filbre spÈciale
gÈomÈtrieque $X_{I,\bar s}$ de $X_I$, i.e. le sous-schÈma dont la
partie connexe du groupe de Barsotti-Tate en chacun de ses points gÈomÈtriques
est de rang $\geq h$ (resp. Ègal ‡ $h$).
\end{defi}

\begin{prop} (cf. \cite{ito2})
Les $X_{I,\bar s}^{=h}$ sont affines de pure dimension $d-h$.
\end{prop}

\rem ces $X^{\geq h}_{\IC,\bar s}$ munissent $X_{\IC,\bar s}$ d'une stratification $\SF$
dite de Newton qui sera la seule considÈrÈe de sorte que l'on fera disparaitre le symbole
$\SF$ de toutes les notations.

\begin{prop} (cf. \cite{h-t} p.116)
Pour tout $1 \leq h < d$, les strates $X_{I,\bar s}^{=h}$ sont gÈomÈtriquement induites sous
l'action du parabolique $P_{h,d-h}(\OC_v)$, au sens o˘
il existe un sous-schÈma fermÈ $X_{I,\bar s,1}^{=h}$ tel que:
$$X_{I,\bar s}^{=h} \simeq X_{I,\bar s,1}^{=h} \times_{P_{h,d-h}(\OC_v/v^{m_1(I)})} 
GL_d(\OC_v/v^{m_1(I)}).$$
\end{prop}

\begin{nota}
On notera $X_{I,\bar s,1}^{\geq h}$ l'adhÈrence de $X_{I,\bar s,1}^{=h}$ dans
$X_{I,\bar s}^{\geq h}$.
\end{nota}

\rem l'action de $P_{h,d-h}(F_v)$ sur $X_{\IC,\bar s,1}^{=h}$ 
se factorise par l'application $\left (
\begin{array}{cc} g_v^c & * \\ 0 & g_v^{et} \end{array} \right ) \mapsto (v(\det
g_v^c),g_v^{et})$. En outre l'action d'un ÈlÈment $w_v \in W_v$ est donnÈe par celle 
de $-\deg(w_v)$ sur le facteur $\Zm$ ci-dessus, o˘ $\deg$ est la composÈe du
caractËre non ramifiÈ de $W_v$, qui envoie les Frobenius gÈomÈtriques sur les uniformisantes, avec la valuation $v$ de $F_v$.

\begin{nota}
On notera comme prÈcÈdemment
$$i_{h}:X_{\IC,\bar s}^{\geq h} \hookrightarrow X_{\IC,\bar s}=X_{\IC,\bar s}^{\geq 1}, \quad
j^{\geq h}:X_{\IC,\bar s}^{=h} \hookrightarrow X_{\IC,\bar s}^{\geq h}, \quad j^{\geq h}_{1}:X_{\IC,\bar s,1}^{=h}
\hookrightarrow X_{\IC,\bar s,1}^{\geq h} \hbox{ et } j^{=h}:=i_h \circ j^{\geq h}.$$
\end{nota}

\subsection{SystËmes locaux d'Harris-Taylor}
\label{para-systeme-locaux}

Soit $D_{v,d}$ l'algËbre ‡ division centrale sur $F_v$ d'invariant $1/d$
et d'ordre maximal $\DC_{v,d}$. ¿ toute reprÈsentation irrÈductible
admissible $\tau_v$ de $\DC_{v,h}^\times$, Harris et Taylor associent 
un systËme local $\FC_{\tau_v,\IC,1}$ sur $X_{\IC,\bar s,1}^{=h}$ muni d'une action de
$G(\Am^{\oo,p}) \times \Qm_p^\times \times P_{h,d-h}(F_v) \times 
\prod_{i=2}^r (B_{v_i}^{op})^\times  \times \Zm$
qui d'aprËs \cite{h-t} p.136, se factorise par $G^{(h)}(\Am^\oo)/\DC_{F_v,h}^\times$ via
\addtocounter{smfthm}{1}
\begin{equation}  \phantomsection \label{eq-action-tordue}
(g^{\oo,p},g_{p,0},c,g_v,g_{v_i},k) \mapsto (g^{p,\oo},g_{p,0}q^{k-v (\det g_v^c)}, 
g_v^{et},g_{v_i}, \delta).
\end{equation}
o˘ $G^{(h)}(\Am^\oo):=G(\Am^{\oo,p}) \times \Qm_p^\times \times GL_{d-h}(F_v) \times
\prod_{i=2}^r (B_{v_i}^{op})^\times \times D_{F_v,h}^\times$, et o˘
$g_v=\left ( \begin{array}{cc} g_v^c & * \\ 0 & g_v^{et} \end{array} \right )$ et
$\delta \in D_{v,h}^\times$ sont tels que $v(\rn (\delta))=k+v(\det g_v^c)$.
On note $\FC_{\tau_v,\IC}$ le faisceau sur $X_{\IC,\bar s}^{=h}$ induit
$$\FC_{\tau_v,\IC}:= \FC_{\tau_v,\IC,1} \times_{P_{h,d-h}(F_v)} GL_d(F_v).$$

\begin{nota}
Toute reprÈsentation irrÈductible $\tau_v$ de $D_{v,h}^\times$ 
est l'image par la correspondance de Jacquet-Langlands locale d'une reprÈsentation
de la forme $\st_t(\pi_v)$ o˘ $\pi_v$ est une reprÈsentation irrÈductible admissible
cuspidale de $GL_g(F_v)$ avec $h=tg$. On la notera $\pi_v[t]_D$ et pour
$(\pi_v[t]_D)_{|\DC_{v,h}^\times}=\bigoplus_{i=1}^{e_{\pi_v}}
\rho_{v,i}$ avec $\rho_{v,i}$ irrÈductible, soient
$$\FC(\pi_v,t)_1=\bigoplus_{i=1}^{e_{\pi_v}}\FC_{\rho_{v,i},\IC,1} \quad \hbox{ et } \quad
\FC(\pi_v,t):= \FC(\pi_v,t)_1 \times_{P_{tg,d-tg}(F_v)} GL_d(F_v)$$
le faisceau sur $X_{\IC,\bar s,1}^{=tg}$ et sa version induite sur $X_{\IC,\bar s}^{=tg}$.
\end{nota}

\begin{defi} \phantomsection \label{defi-fht}
(\textbf{SystËmes locaux dits d'Harris-Taylor})
Pour $\Pi_t$ une reprÈsentation de $GL_{tg}(F_v)$, on note 
$HT(\pi_v,\Pi_t):=\FC(\pi_v,t)[d-tg] \otimes \Xi^{\frac{tg-d}{2}} \otimes \Pi_t,$
o˘ l'action se dÈduit par induction par celle de
$(g^p,g_{p,0},c,g_v^{et},g_{v_i},g_v^c,\sigma)$ dans
$$G(\Am^{\oo,p}) \times \Qm_p^\times
\times (\Zm \times GL_{d-tg}(F_v))^+ \times \prod_{i=2}^r (B_{v_i}^{op})^\times \times GL_{tg}(F_v) \times W_v$$
sur le faisceau non induit, o˘ le radical unipotent de $P_{tg,d}(F_v)$ agit trivialement,
$g_v^c$ agit sur $\Pi_t$ et pour $\gamma \in D_{v,tg}^\times/\DC_{v,tg}^\times$ 
tel que $v(\rn \gamma)=v(\det g_v^c) -\deg \sigma$,
$$(g^p,g_{p,0}q^{-c+ v(\det g_v^c) - \deg \sigma}, \gamma,g^{et},g_{v_i}) \in 
G^{(tg)}(\Am^\oo)/\DC_{v,tg}^\times$$ 
agit sur $\FC(t,\pi_v)$.
\end{defi}

\rem les $\FC(\pi_v,t)$ n'Ètant pas irrÈductibles,
$j^{\geq tg}_{!*} HT(\pi_v,\Pi_t)$ n'est pas irrÈductible. 

Soit alors $\pi_v$ une reprÈsentation irrÈductible admissible cuspidale de $GL_g(F_v)$
et on note $s=\lfloor \frac{d}{g} \rfloor$. Pour tout $1 \leq t \leq s$, considÈrons
la filtration de stratification du \S \ref{para-def-filtration0}
$$0=\Fil^{-d}_*(\pi_v,\Pi_t) \subset \Fil^{1-d}_*(\pi_v,\Pi_t) \subset \cdots \subset
\Fil^0_*(\pi_v,\Pi_t)=j^{=tg}_{!} HT(\pi_v, \Pi_t).$$

\begin{prop} 
La filtration $\Fil^\bullet_*(\pi_v,\Pi_t)$ est adaptÈe ‡ la stratification de Newton au sens
de la dÈfinition \ref{defi-adpate}. 
Ses graduÈs $\gr^{-r}_*(\pi_v,\Pi_t):=\Fil^{-r}_*(\pi_v,\Pi_t)/\Fil^{-r-1}_*(\pi_v,\Pi_t)$ sont tous
nuls sauf pour $r=kg-1$ pour $t \leq k \leq s$ auquel cas 
$$\gr^{1-kg}_*(\pi_v,\Pi_t) \simeq  j^{= kg}_{!*} 
HT (\pi_v,\Pi_t \overrightarrow{\times} \st_{k-t}(\pi_v)) \otimes \Xi^{(t-k)/2}.$$
\end{prop}

\begin{proof}
On rappelle que d'aprËs \cite{boyer-invent2} proposition 4.3.1 et corollaire 5.4.1,
pour tout $1 \leq t \leq s$, les constituants simples de $j^{= tg}_{!} HT(\pi_v,\Pi_{t})$ sont les 
$$P_k:= j^{= kg}_{!*} 
HT (\pi_v,\Pi_t \overrightarrow{\times} \st_{k-t}(\pi_v)) \otimes \Xi^{(t-k)/2},$$
pour $t \leq k \leq s$.
Pour tout $tg \leq  r \leq d$, l'image $\CoFil_{\SF,*,r}(\pi_v,\Pi_t)$ du morphisme
d'adjonction
$$j^{= tg}_{!} HT(\pi_v, \Pi_t) =\lexp p j^{1 \leq r}_!  j^{1 \leq r,*}
\bigl ( j^{= tg}_{!} HT(\pi_v, \Pi_t) \bigr )
\longrightarrow \lexp p j^{1†\leq r}_* j^{1 \leq r,*}
\bigl ( j^{= tg}_{!} HT(\pi_v, \Pi_t) \bigr )$$
est, par dÈfinition de la notion d'extension intermÈdaire, 
$j^{1 \leq r}_{!*} j^{1 \leq r,*} \bigl ( j^{= tg}_{!} HT(\pi_v, \Pi_t) \bigr )$.
Comme le noyau de ce morphisme d'adjonction est ‡ support dans 
$X^{\geq r+1}_{\IC,\bar s}$, on en dÈduit que l'image de $\CoFil_{\SF,*,r}(\pi_v,\Pi_t)$
dans le groupe de Grothendieck est supÈrieure ou Ègale ‡
$\sum_{k=t}^{\lfloor \frac{r}{g} \rfloor} P_k$ et l'ÈnoncÈ consiste ‡ montrer l'ÈgalitÈ. 
Soit alors $k_0$ maximal tel que $P_{k_0}$ et $\CoFil_{\SF,*,r}(\pi_v,\Pi_t)$ aient un
constituant $Z$ en commun lequel sera donc pur de poids $t-k_0$ puisque les $P_k$ 
sont purs de poids $t-k$. La filtration par les poids de $\CoFil_{\SF,*,r}(\pi_v,\Pi_t)$
nous donne alors que $Z$ est un facteur direct du socle de
$\CoFil_{\SF,*,r}(\pi_v,\Pi_t)$.
Or celui-ci est le $j^{1 \leq r}_{!*}$ du socle de $j^{1 \leq r,*} \bigl ( 
j^{= tg}_{!} HT(\pi_v, \Pi_t) \bigr )$, et donc somme directe de certains $P_k$
pour $1 \leq k \leq \lfloor \frac{r}{g} \rfloor$ de sorte que $k_0=\lfloor \frac{r}{g} \rfloor$
et $Z=P_{k_0}$.
\end{proof}

\rem le raisonnement prÈcÈdent se gÈnÈralise au cas d'un faisceau pervers dont les
constituants irrÈductibles sont ordonnÈs simultanÈment par la dimension de leur
support et leur poids, i.e. ceux dont la dimension du support est la plus grande sont 
ceux de plus haut poids. Dualement si ceux de plus haut poids sont ceux dont
la dimension du support est minimale, le raisonnement prÈcÈdent s'adapte aux
$\Fil^\bullet_!$. 

\begin{prop} \label{prop-fp-K1}
La filtration de stratification $\Fil^\bullet_! \bigl ( \Fil^{-tg}_*(\pi_v,\Pi_t) \bigr )$ de
$$\Fil^{-tg}_*(\pi_v,\Pi_t)= \ker_\FC \bigl ( 
j^{= tg}_! HT(\pi_v,\Pi_t) \twoheadrightarrow  j^{= tg}_{!*} HT(\pi_v,\Pi_t) \bigr )$$
est triviale, i.e.
tous ses graduÈs sont nuls sauf $\gr^{(t+1)g}_! \bigl ( \Fil^{-tg}_*(\pi_v,\Pi_t) \bigr ) \simeq
\Fil^{-tg}_*(\pi_v,\Pi_t)$.
\end{prop}

\begin{proof}
Il s'agit donc de montrer que le morphisme d'adjonction
$$j^{= (t+1)g}_! j^{= (t+1)g,*} \bigl ( \Fil^{-tg}_*(\pi_v,\Pi_t) \bigr ) \longrightarrow 
\Fil^{-tg}_*(\pi_v,\Pi_t)$$ 
est surjectif. Raisonnons par l'absurde et soit $k_0\geq t+2$ minimal tel que, 
avec les notations
prÈcÈdentes, $P_{k_0}$ et le conoyau de ce morphisme d'adjonction aient un constituant $Z$
en commun. Ainsi $Z$ est un facteur direct de $j^{1 \leq k_0g,*} \bigl ( \Fil^{-tg}_*(\pi_v,\Pi_t) 
\bigr )$, puisqu'il est, par construction, un quotient, et un sous-objet via la filtration par les
poids. Pour $z$ un point gÈnÈrique de $X^{=k_0g}_{\IC,\bar s}$, on en dÈduit donc
que la fibre en $z$ de $\hi^{-\dim z}  \bigl ( \Fil^{-tg}_*(\pi_v,\Pi_t) \bigr )$ est non nulle.
Comme la fibre en $z$ de $\hi^{-\dim z} j^{\geq tg}_! HT(\pi_v,\Pi_t)$ est nulle, on en
dÈduit alors que la fibre en $z$ de $\hi^{-\dim z-1} j^{\geq tg}_{!*} HT(\pi_v,\Pi_t)$ 
est non nulle ce qui, comme $k_0 \geq t+2$, contredit le rÈsultat principal de 
\cite{boyer-invent2} que nous rappelons au thÈorËme \ref{theo-boyer} ci-aprËs.
\end{proof}

\rem on pourrait montrer de mÍme que $\Fil^p_! (\Fil^q_*(\pi_v,\Pi_t))$ est soit nul
soit Ègal ‡ $\Fil^{q}_*(\pi_v,\Pi_t)$. 


\begin{theo}  \phantomsection\label{theo-boyer} (cf. \cite{boyer-invent2} thÈorËme 2.2.5) 
La fibre en un point gÈnÈrique de $X^{=h}_{\IC,\bar s}$ de $\hi^i j^{\geq tg}_{!*} HT(\pi_v,\Pi_t)$ est nulle sauf pour $(h,i)=\bigl ( (t+r)g,(t+r)g-d-r \bigr )$ avec
$0 \leq r \leq \lfloor \frac{d}{g} \rfloor-t$. Dans ce cas,
dans le groupe de Grothendieck des reprÈsentations de $GL_{(t+r)g}(F_v) \times \Zm$, elle est
un multiple de $\Pi_t \overrightarrow{\times} \speh_{r}(\pi_v) \otimes \Xi^{-\frac{(t+r)g-d-r}{2}}$.
\end{theo}

\rem au \S \ref{para-recurrent}, nous proposerons une dÈmonstration 
cohomologique de la proposition
prÈcÈdente ainsi qu'une nouvelle preuve du thÈorËme ci-dessus.

En itÈrant la proposition prÈcÈdente, dans la filtration exhaustive de stratification 
$$0=\Fill^{2^{d-1}}_!(\pi_v,\Pi_t) \subset \cdots \subset \Fill^0_!(\pi_v,\Pi_t) \subset \cdots 
\subset \Fill^{2^{d-1}}_!(\pi_v,\Pi_t)=j^{= tg}_{!} HT(\pi_v, \Pi_t),$$
l'ordre d'apparition des $P_k$ est aussi donnÈ par les poids. Les indices
prÈcis des graduÈs non triviaux Ètant plus complexes nous aurons besoin des notations 
suivantes.

\begin{nota}  \phantomsection \label{nota-gr}
Pour tout $h \geq 2$ on note $r_h=2^{d-2}+\cdots +2^{d-h}$ et on pose $r_1=0$.
Pour $k \geq 1$ on note $\delta_k(g)=\sum_{i=1}^{k} 2^{-ig}$ et on pose $\delta_0=0$.
\end{nota}

\begin{coro} \phantomsection \label{coro-fp-K1}
Les graduÈs $\grr^r_!(\pi_v,\Pi_t):=\Fill^r_!(\pi_v,\Pi_t)/\Fill^{r-1}_!(\pi_v,\Pi_t)$ sont tous nuls
sauf ceux de la forme $r=r_{tg}-2^{d-t}\delta_k(g)$ pour $k=0,\cdots,s-1$ auquel cas 
$$\grr^{r_{tg}-2^{d-t}\delta_k(g)}_!(\pi_v,\Pi_t) \simeq j^{= (t+k)g}_{!*} 
HT (\pi_v,\Pi_t \overrightarrow{\times} \st_k(\pi_v)) \otimes \Xi^{-k/2}.$$
\end{coro}

\rem en utilisant que l'extension intermÈdiaire est le top de l'extension par zÈro,
on notera que $\Fil^\bullet_*(\pi_v,\Pi_t)$ et $\Fill^\bullet_!(\pi_v,\Pi_t)$ coÔncident, ‡
la numÈrotation prËs, avec la filtration par les socles de $j^{=tg}_! HT(\pi_v,\Pi_t)$.

\subsection{Faisceau pervers de Hecke des cycles proches}
\label{para-psi0}

\begin{defi}
Pour tout $J \in \IC$, les faisceaux pervers des cycles Èvanescents
$R\Psi_{\bar \eta_v,J}(\overline \Qm_l)[d-1](\frac{d-1}{2})$ sur $X_{J,\bar s}$
dÈfinissent un $W_v$-faisceau pervers de Hecke, au sens de la dÈfinition 1.3.6
de \cite{boyer-invent2}, que l'on note $\Psi_{\IC}$.
\end{defi}

\rem dans \cite{boyer-invent2} \S 2.2, on utilise l'action de la monodromie pour
dÈcouper $\Psi_{\IC}$ 
$$\Psi_{\IC} = \bigoplus_{\atop{1 \leq g \leq d}{\pi_v \in \cusp_v(g)}} \Psi_{\IC,\pi_v}$$
o˘ $\cusp_v(g)$ dÈsigne l'ensemble des classes d'Èquivalence inertielles, cf. dÈfinition 1.1.3
de \cite{boyer-invent2}, des reprÈsentations irrÈductibles cuspidales
de $GL_g(F_v)$ avec $1 \leq g \leq d$.

\begin{defi} (\textbf{Faisceau pervers dits d'Harris-Taylor})
On note $P(t,\pi_v)$ le $W_v$-faisceaux pervers de Hecke sur $X_{\IC,\bar s}^{\geq 1}$ de support
$X_{\IC,\bar s}^{\geq h}$ et de poids zÈro dÈfini par
$$j^{= tg}_{!*} HT(\pi_v,\st_t(\pi_v)) \otimes \LF(\pi_v),$$
o˘ $\LF^\vee$ est la correspondance de Langlands, cf. \cite{h-t} et o˘
l'action de $G(\Am^\oo) \times W_v$ sur $P(t,\pi_v)$ se dÈfinit par induction en faisant agir
$$(g^p,g_{p,0},g_v^c,g_v^{et},g_{v_i},\sigma) \in G(\Am^{\oo,p}) \times \Qm_p^\times \times GL_{tg}(F_v) \times GL_{d-tg}(F_v)
\times \prod_{i=2}^r (B_{v_i}^{op})^\times \times W_v$$
via l'action de:
\begin{itemize}
\item $(g^p,g_{p,0}q^{-\deg \sigma},\gamma,g_v^{et},g_{v_i}) \in G^{(tg)}(\Am^{\oo})/\DC_{v,tg}^\times$
sur $\FC(t,\pi_v)$ o˘ $\gamma \in D_{v,tg}^\times$ est tel que $v(\rn \delta)=v(\det g_v^c)-\deg \sigma$;

\item $(g_v^c,\sigma)$ sur $\st_t(\pi_v) \otimes \LF(\pi_v)$.
\end{itemize}
\end{defi}

\rem on notera que $P(t,\pi_v)$ ne dÈpend, en tant que
$W_v$-faisceau pervers de Hecke, que de la classe d'Èquivalence inertielle de $\pi_v$.
D'aprËs la remarque qui suit le thÈorËme 2.4.4 de \cite{boyer-invent2},
les faisceaux pervers d'Harris-Taylor $P(t,\pi_v)$ sont de la
forme $e_{\pi_v} \PC(t,\pi_v)$ o˘ $\PC(t,\pi_v)$ est un faisceau pervers simple,
o˘ $e_{\pi_v}$ est le cardinal de la classe d'Èquivalence inertielle de $\pi_v$, cf. loc. cit. dÈfinition 1.1.3. Pour ce qui concerne les groupes de Grothendieck de faisceaux pervers
dits de Hecke, on renvoie le lecteur au \S 7 de \cite{boyer-invent2}.

\begin{prop} \phantomsection \label{prop-psi-fil}
Soit 
$$0=\Fil^0_!(\Psi_{\IC,\pi_v}) \subset \Fil^1_!(\Psi_{\IC,\pi_v}) \subset \cdots \subset
\Fil^d_!(\Psi_{\IC,\pi_v})=\Psi_{\IC,\pi_v}$$
la filtration de stratification de $\Psi_{\IC,\pi_v}$ de la proposition \ref{prop-ss-filtration}. 
Pour tout $r$ non divisible par $g$, le graduÈ $\gr^r_!(\Psi_{\IC,\pi_v})$ est nul et pour
$r=tg$ avec $1 \leq t \leq s$, il est ‡ support dans $X^{\geq tg}_{\IC,\bar s}$ avec
$j^{\geq tg,*} \gr^{tg}_!(\Psi_{\IC,\pi_v}) \simeq HT(\pi_v, \st_t(\pi_v)) \otimes 
L_g(\pi_v) (\frac{1-t}{2})$. La surjection
$$j^{= tg}_! HT(\pi_v, \st_t(\pi_v)) \otimes L_g(\pi_v) (\frac{1-t}{2}) \twoheadrightarrow 
\gr^{tg}_!(\Psi_{\IC,\pi_v}),$$
a alors pour image dans le groupe de Grothendieck
$\sum_{i=t}^s \PC(i,\pi_v)(\frac{1+i-2t}{2}).$
\end{prop}

\begin{proof}
D'aprËs \cite{boyer-invent2} corollaire 5.4.2, dans le groupe de Grothendieck on a
$$[\Psi_{\IC,\pi_v}]= \sum_{t=1}^{s} \bigl ( \sum_{i=t}^s \PC(i,\pi_v)(\frac{1+i-2t}{2}) \bigr ).$$
Ainsi par construction, les $\gr^r_!(\Psi_{\IC,\pi_v})$ sont nuls si $r$ n'est pas de la forme $kg$
pour $1 \leq k \leq s$. Par ailleurs comme les constituants de $j^{=kg,*} \Psi_{\IC,\pi_v}$
sont les $j^{=kg,*} \PC(i,\pi_v) (\frac{1+i-2t}{2})$ pour $1 \leq t \leq i \leq k$, on 
en dÈduit que les constituants communs ‡ $\Psi_{\IC,\pi_v}$ et ‡ 
$j^{= kg}_! j^{=kg,*} \Psi_{\IC,\pi_v}$ sont les $\PC(i,\pi_v)  (\frac{1+i-2t}{2})$
pour $1 \leq t \leq k$ et $t \leq i \leq s$ de sorte que, dans le groupe de Grothendieck,
l'image $\Fil^{kg}_!(\Psi_{\IC,\pi_v})$ de
$j^{= kg}_! j^{=kg,*} \Psi_{\IC,\pi_v} \longrightarrow \Psi_{\IC,\pi_v}$
est infÈrieure ‡ $\sum_{t=1}^k \sum_{i=t}^s \PC(i,\pi_v)(\frac{1+i-2t}{2})$.
Supposons par l'absurde qu'il existe $k$ tel que $\Fil^{kg}_!(\Psi_{\IC,\pi_v})$ 
soit strictement infÈrieure ‡ la somme ci-dessus. Soient alors $i_0>k$ minimal tel que 
$\PC(i_0,\pi_v)(\frac{1+i-2t}{2})$  ne soit pas un  constituant de $\Fil^{kg}_!(\Psi_{\IC,\pi_v})$ 
et $z$ un point gÈnÈrique de $X^{=i_0g}_{\IC,\bar s}$. D'aprËs \ref{theo-boyer}, dans le 
groupe de Grothendieck des 
reprÈsentations de $GL_{i_0g}(F_v) \times W_v$, la fibre en $z$ de 
$\hi^{-\dim z} \PC(i_0,\pi_v) (\frac{1+i_0-2t}{2})$ (resp. de $\hi^{-\dim z-1} \PC(i_0-1,\pi_v)
(\frac{i_0-2t}{2})$) est supÈrieure ‡
$\st_{i_0}(\pi_v) \otimes L_g(\pi_v) (\frac{i_0g-d+1+i_0-2t}{2})$; 
par ailleurs ce n'est jamais le cas
pour tout autre $\hi^{-\dim z-\delta} \PC(i,\pi_v)(\frac{1+i-2t'}{2})$ quel que soit 
$\delta \geq 0$. Comme par hypothËse $\PC(i_0-1,\pi_v) (\frac{i_0-2t}{2})$
est un constituant de $\Fil^{kg}_! (\Psi_{\IC,\pi_v})$ et que 
$\PC(i_0,\pi_v) (\frac{1+i_0-2t}{2})$ n'en est pas un, on en dÈduit que
$\hi^{-\dim z} \bigl ( \Psi_{\IC,\pi_v}/ \Fil^{kg}_! (\Psi_{\IC,\pi_v}) \bigr )$, et donc 
$\hi^{-\dim z} \Psi_{\IC,\pi_v}$, est supÈrieure ‡ 
$\st_{i_0}(\pi_v) \otimes L_g(\pi_v) (\frac{i_0g-d+1+i_0-2t}{2})$,
ce qui n'est pas d'aprËs le calcul des germes des faisceaux de cohomologie de $\Psi_\IC$
donnÈ au corollaire 2.2.10 de \cite{boyer-invent2}.
\end{proof}

\rem la figure \ref{figure3} illustre la filtration de stratification
de $\Psi_{\IC,\pi_v}$ o˘ chaque cercle reprÈsente un faisceau pervers
d'Harris-Taylor $\PC(t,\pi_v)(\frac{1-t}{2}+k)$. 
L'utilisation de \ref{theo-boyer} et le corollaire 2.2.10 dans la preuve prÈcÈdente
revient ‡ dire que pour tout
$1 \leq t \leq s-1$ et pour tout $0 \leq k \leq s-t-1$, dans toute filtration
croissante de $\Psi_{\IC,\pi_v}$ l'indice $k_1$ du graduÈ contenant
$\PC(t-k,\pi_v)(\frac{t-k-1}{2})$ est infÈrieur ou Ègale ‡ l'indice $k_2$
du graduÈ contenant $\PC(t-k+1,\pi_v)(\frac{t-k-2}{2})$. Le lecteur pourra trouver
une illustration graphique de ces contraintes ‡ la figure \ref{figure2} o˘ lorsque 
deux cercles sont reliÈs il faut que l'indice de celui qui est le plus bas soit supÈrieur ou 
Ègal ‡ l'autre. Au \S \ref{para-psi} nous proposerons une preuve cohomologique de la 
proposition
\ref{prop-psi-fil} et nous en dÈduirons une nouvelle dÈmonstration du calcul des 
$\hi^i \Psi_{\IC}$.

\begin{figure}[ht]
\centering
\scalebox{.72}{\input{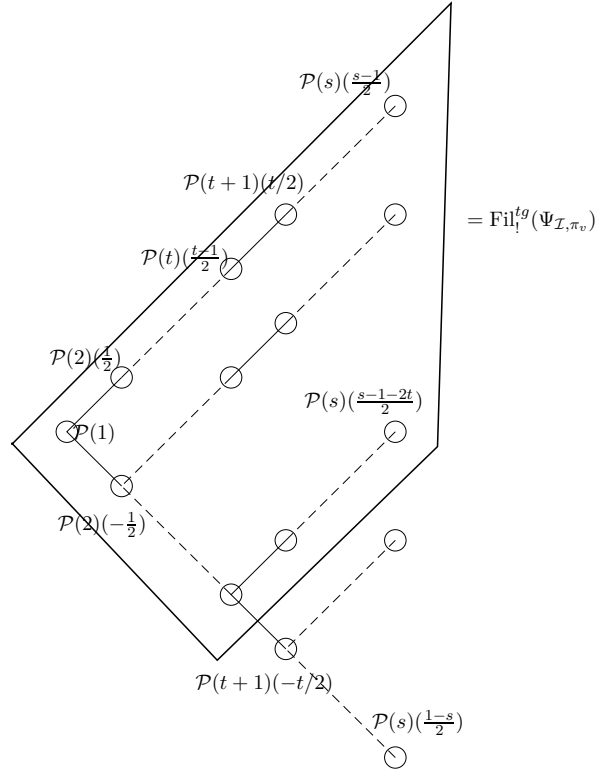}}
\caption{\label{figure3} Illustration de la filtration de
stratification de $\Psi_{\IC,\pi_v}$.}
\end{figure}

\begin{figure}[ht]
\centering
\scalebox{.8}{\input{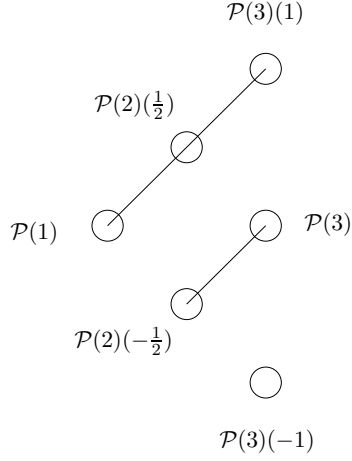}}
\caption{\label{figure2} Illustration des contraintes
d'une filtration de $\Psi_{\IC,\pi_v}$ avec $s=3$.}
\end{figure}

Le lecteur notera qu'en supprimant les graduÈs nuls de $\Fil^\bullet_!(\Psi_{\IC})$,
la filtration obtenue coÔncide avec la filtration par les noyaux itÈrÈs de la monodromie. 
Par dualitÈ, la filtration $\Fil^\bullet_*(\Psi_\IC)$ ‡ laquelle on supprime les graduÈs nuls,
coÔncide avec celle par les images itÈrÈes de la monodromie. Ainsi la bifiltration 
de monodromie s'obtient en considÈrant les $\Fil^i_* \bigl ( \Fil^j_! (\Psi_\IC) \bigr )$.
PrÈcisÈment, en couplant les propositions \ref{prop-fp-K1} et \ref{prop-psi-fil}, 
on obtient la description suivante.

\begin{coro}
Avec les notations prÈcÈdentes les bigraduÈs $\gr^i_* \bigl ( \gr^j_! (\Psi_{\IC,\pi_v} ) \bigr )$
(resp. $\gr^i_! \bigl ( \gr^j_* (\Psi_{\IC,\pi_v} ) \bigr )$)
de la bifiltration $\Fil^i_* \bigl ( \Fil^j_! (\Psi_\IC) \bigr )$ 
(resp. de $\Fil^i_! \bigl ( \Fil^j_* (\Psi_\IC) \bigr )$) sont nuls sauf pour $(i,j)$
de la forme $(t'g,tg)$ avec $1 \leq t \leq t' \leq s$ auquel cas 
$$\gr^{t'g}_* \bigl ( \gr^{tg}_! (\Psi_{\IC,\pi_v} ) \bigr ) \simeq \PC(t',\pi_v)(\frac{1-2t+t'}{2}),
\quad \Bigl ( \hbox{resp. } \gr^{t'g}_! \bigl ( \gr^{tg}_* (\Psi_{\IC,\pi_v} ) \bigr ) \simeq 
\PC(t',\pi_v)(\frac{2t-1-t'}{2}) \Bigr ).$$
\end{coro}

\rem pour $s=3$ et $g=1$, le lecteur trouvera ‡ la figure \ref{fig-filtration} une illustration 
des $\Fil^\bullet_!(\Psi_{\IC,\pi_v})$ et $\Fil^\bullet_*(\Psi_{\IC,\pi_v})$.

En ce qui concerne la filtration exhaustive de stratification de $\Psi_{\IC,\pi_v}$,
avec les notations de \ref{nota-gr}, la chasse aux indices fournit le corollaire suivant.

\begin{coro}
Les graduÈs de la filtration exhaustive de stratification de  $\Psi_{\IC,\pi_v}$
$$0=\Fill^{2^{d-1}}_!(\Psi_{\IC,\pi_v}) \subset \cdots \subset \Fill^0_!(\Psi_{\IC,\pi_v}) \subset
\cdots \subset \Fill^{2^{d-1}}_!(\Psi_{\IC,\pi_v})=\Psi_{\IC,\pi_v}$$
sont nuls pour $r$ qui n'est pas de la forme $i_{tg}-2^{d-t} \delta_k(g)$ pour 
$1 \leq t+k \leq s$ avec $t \geq 1$ et $k \geq 0$, et sinon isomorphe ‡
$\grr^{i_{tg}-2^{d-t} \delta_k(g)}_!(\Psi_{\IC,\pi_v}) \simeq \PC(t+k,\pi_v)(\frac{k+1-2t}{2}).$
\end{coro}

\begin{figure}[ht]
\centering
\scalebox{.8}{\input{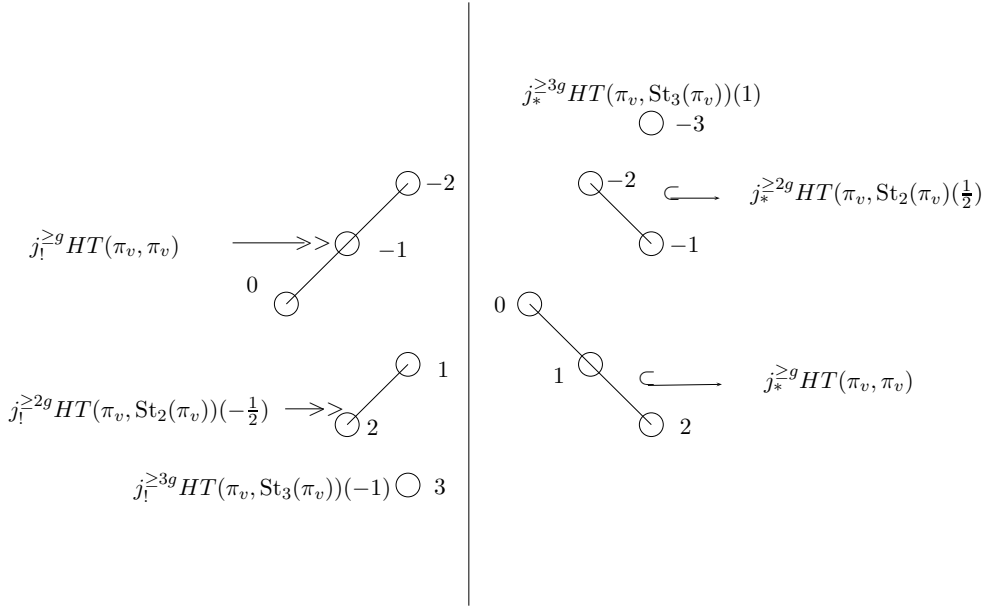}}
\caption{\label{fig-filtration} Illustration de la filtration (fig. ‡ gauche)
et de la cofiltration (fig. ‡ droite) de stratification de $\Psi_{\IC,\pi_v}$ avec $s=3$.}
\end{figure}

\appendix

\section{Retour sur les rÈsultats faisceautiques de \texorpdfstring{\cite{boyer-invent2}}{Lg}}
\label{para-retour}

L'objectif de cet appendice est de revenir sur la preuve des principaux rÈsultats
de \cite{boyer-invent2} en utilisant les filtrations de stratification, l'intÈrÍt Ètant de
simplifier les arguments et de dÈgager une future stratÈgie d'Ètude sur
$\overline \Zm_l$. Rappelons tout d'abord la stratÈgie de loc. cit.

\textit{\'Etape 1}: 
on utilise tout d'abord le thÈorËme de comparaison de Berkovich-Fargues afin de
relier les germes des faisceaux de cohomologie du complexe des cycles Èvanescents
‡ la cohomologie des espaces de Lubin-Tate.

\textit{\'Etape 2}: dans \cite{h-t}, les auteurs donnent une \og recette \fg{} afin de
calculer la somme alternÈe des groupes de cohomologie des $j_!^{\geq tg} HT(\pi_v,\Pi_t)$
avec les notations prÈcÈdentes. Ainsi au terme du \S 5.4 de \cite{boyer-invent2},
on connait les constituants simples de $j^{\geq tg}_! HT(\pi_v,\Pi_t)$ et 
$\Psi_{\IC,\pi_v}$.

\textit{\'Etape 3}: 
il s'agit alors de calculer les faisceaux de cohomologie des faisceaux pervers
d'Harris-Taylor ainsi que ceux du complexe des cycles Èvanescents. Pour ce faire:
\begin{itemize}
\item dans \cite{boyer-invent2}, on utilise une propriÈtÈ d'autodualitÈ ‡ la Zelevinski
sur la cohomologie des espaces de Lubin-Tate, prouvÈe par Fargues;

\item une deuxiËme solution
est proposÈe dans \cite{boyer-compositio} et repose sur le thÈorËme de Lefschetz difficile
et sur des calculs fastidieux de groupes de cohomologie.
\end{itemize}
Le but de cette section est de donner une preuve alternative de l'Ètape 3 en utilisant
les filtrations de stratification; en particulier nous n'utilisons plus \ref{theo-boyer} et en 
proposons mÍme une nouvelle dÈmonstration.

\subsection{Rappels cohomologiques de \texorpdfstring{\cite{h-t}}{Lg}}
\label{para-rappel-coho}

Dans ce paragraphe nous allons rappeler, en utilisant les Ètapes 1 et 2 ci-avant,
lesquels des rÈsultats de \cite{boyer-compositio} nous utiliserons. Notons que ceux-ci
dÈcoulent des techniques habituelles de formules des traces.

On note $H^i_{\bar \eta}$ le $i$-Ëme groupe de cohomologie de la variÈtÈ de Shimura
$X_{\IC,\bar \eta}$. Pour $\Pi^\oo$ une reprÈsentation irrÈductible de $G(\Am^\oo)$,
on note $[H^i_{\bar \eta}]\{\Pi^{\oo,v} \}$ la composante $\Pi^{\oo,v}$-isotypique de 
$H^i_{\bar \eta}$ dans le groupe de Grothendieck des 
$GL_d(F_v) \times W_v$-reprÈsentations. Rappelons que si celle-ci est non nulle
alors $\Pi^{\oo,v}$ est la composante hors $\oo,v$ d'une reprÈsentation automorphe
cohomologique $\Pi$. On fixe alors jusque la fin de cet appendice 
une telle reprÈsentation automorphe $\Pi$ telle que 
sa composante locale $\Pi_v \simeq \speh_s(\pi_v)$ pour $\pi_v$ une reprÈsentation 
irrÈductible admissible cuspidale de $GL_g(F_v)$ avec $d=sg$.

\begin{lemm}(cf. la proposition 3.5.1 de \cite{boyer-compositio})  \phantomsection
\label{lem-rappel-coho} 
Pour tout $1 \leq t \leq s$, en tant que reprÈsentation de $GL_d(F_v)$,
$[H^{t-s}(j^{\geq tg}_{!*} HT(\pi_v,\Pi_t))]\{ \Pi^{\oo,v} \}$
est Ègale ‡ un multiple de 
$\Pi_t \overrightarrow{\times} \speh_{s-t}(\pi_v)$.
\end{lemm}

\rem rappelons que $[H^*(j^{\geq tg}_! HT(\pi_v, \st_{t}(\pi_{v})))] \{ \Pi^{\oo,v}\}$
est calculÈe dans \cite{h-t} de sorte que le calcul des $[H^i(j^{\geq tg}_{!*} HT(\pi_v,\Pi_t))]$
se dÈduit, en utilisant la puretÈ, de l'ÈgalitÈ de \cite{boyer-invent2}
\addtocounter{smfthm}{1}
\begin{equation}  \phantomsection \label{eq-egalite}
j^{= tg}_{!*} HT(\pi_v,\st_{t}(\pi_{v}))= \sum_{r=0}^{s-t} (-1)^r 
j^{= (t+r)g}_{!} HT \Bigl ( \pi_v,\st_{t}(\pi_{v}) \overrightarrow{\times} \speh_r(\pi_v)
\Bigr ) \otimes \Xi^{-r/2}.
\end{equation}

\begin{prop}  \phantomsection \label{prop-rappel-coho}
Dans la suite spectrale des cycles Èvanescents
$E_2^{i,j}=H^i(\hi^j \Psi_\IC) \Rightarrow H^{i+j}_\eta$
si $\st_s (\pi_v) \otimes L_g(\pi_v) (-\frac{k}{2})$ est un 
sous-quotient de $[E_{2}^{i,j}]\{ \Pi^{\oo,v}\}$ avec $i +j>0$ alors
$i+j=1$, $j=-g$ et $k=s-1$.
\end{prop}

\begin{proof}
Via la suite spectrale associÈe ‡ la stratification de Newton, 
on se ramËne ‡ regarder les groupes de cohomologie 
des $(\hi^j \Psi_{\IC,\pi_v})_{|X^{=tg}_{\IC,\bar s}}$
ce qui impose de considÈrer les $j$ de la forme $(t-s)g-r$ avec $1 \leq t \leq s$ et
$0 \leq r \leq t-1$. Comme on regarde $i+j >0$ le cas $t=s$ est exclu.
L'isomorphisme $\bigl ( \hi^{(t-s)g-r} \Psi_{\IC,\pi_v} \bigr )_{|X^{=tg}_{\IC,\bar s}} 
\simeq HT(\pi_v, LT_{\pi_v}(t,r))[(t-s)g]$ nous amËne ‡ regarder, pour $r=0$ et $i>0$, les  
$[H^{i}(j^{\geq tg}_{!} HT(\pi_v, \st_{t}(\pi_{v})))]\{ \Pi^{\oo,v} \}$.
D'aprËs \cite{boyer-invent2} corollaire 5.4.1, l'ÈgalitÈ (\ref{eq-egalite}) s'inverse en
$$j^{= tg}_{!}  HT(\pi_v,\st_{t}(\pi_{v}))=\sum_{k=0}^{s-t}
j^{= (t+k)g}_{!*}  HT(\pi_v,\st_{t}(\pi_{v}) \overrightarrow{\times} \st_{k}(\pi_v)) \otimes
\Xi^{-k/2}$$ 
ce qui nous ramËne ‡ regarder chacun des
$[H^{i}(j^{\geq (t+k)g}_{!*} HT(\pi_v, \st_{t}(\pi_{v}) \overrightarrow{\times} \st_{k}(\pi_v)))]\{ \Pi^{\oo,v} \}$ pour $i>0$.
Le rÈsultat dÈcoule alors du lemme prÈcÈdent.
\end{proof}

\subsection{Retour sur la filtration de stratification des \texorpdfstring{$j^{\geq tg}_! HT(\pi_v,\Pi_t)$}{Lg}}
\label{para-recurrent}

Dans ce paragraphe, on fixe une reprÈsentation admissible irrÈductible cuspidale $\pi_v$
de $GL_g(F_v)$ o˘ $g$ est un diviseur de $d=sg$ et pour $1 \leq t \leq s$,
on se propose tout d'abord de prouver la proposition \ref{prop-fp-K1} 
sans utiliser le thÈorËme \ref{theo-boyer}. 

\rem dans le cas o˘ $g$ ne divise pas $d$, le calcul des faisceaux de cohomologie
des $j^{\geq tg}_{!*} HT(\pi_v,\Pi_t)$ de \cite{boyer-invent2} dÈcoule assez simplement
de l'hypothËse de rÈcurrence sur le modËle local, puisque les points supersinguliers
n'entrent pas en jeu.

Pour ce faire on raisonne par rÈcurrence sur
$t$ de $s$ ‡ $1$. Les cas $t=s$ et $s-1$ Ètant Èvidents, supposons donc le rÈsultat acquis
jusqu'au rang $t+1$ et traitons le cas de $t$. On pose 
$P:=i_{tg+1,*} \lexp p \hi^{-1}_{libre} i_{tg+1}^* j^{\geq tg}_* HT(\pi_v,\Pi_t)$ avec
$$0 \rightarrow P \longrightarrow j_!^{= tg} HT(\pi_v,\Pi_t) \longrightarrow
j_{!*}^{= tg} HT(\pi_v,\Pi_t) \rightarrow 0,$$
et il s'agit de montrer que le morphisme d'adjonction
$j^{\geq (t+1)g}_! j^{\geq (t+1)g,*}P \longrightarrow P$ est surjectif.
Notons alors $P_0$
l'image de ce morphisme.  De la connaissance des constituants irrÈductibles de $P$ et,
d'aprËs l'hypothËse de rÈcurrence, des quotients de $j^{\geq (t+1)g}_! j^{\geq (t+1)g,*}P$,
l'image de $P_0$ dans le groupe de Grothendieck est, en utilisant que les strates non 
supersinguliËres sont gÈomÈtriquement induites, de la forme
$[P_0]=\sum_{i=1}^{s-t-k} j^{= (t+i)g}_{!*} 
HT(\pi_v,\Pi_t \overrightarrow{\times} \st_i(\pi_v)) \otimes \Xi^{-i/2}$
pour un entier $0 \leq k \leq s-t-1$. Par ailleurs la filtration par les poids de $P$
fournit un quotient $P \twoheadrightarrow P'_0$ o˘ $[P'_0]=[P_0]$ dans le groupe
de Grothendieck de sorte que le composÈ $P_0 \hookrightarrow P \twoheadrightarrow P'_0$
est un isomorphisme. On obtient ainsi que $P_0$ est un facteur direct de 
$P \simeq P_0 \oplus Q_0$ et il nous faut alors prouver que $Q_0$ est nul.
Pour ce faire nous allons raisonner sur la cohomologie des systËmes locaux
d'Harris-Taylor, l'idÈe Ètant de montrer qu'il existe un indice
$i \leq -k$ tel que $H^i(j^{\geq tg}_! HT(\pi_v,\Pi_t))$ est non nul de sorte que comme
$j^{\geq tg}$ est affine et donc que tous les $H^i(j^{\geq tg}_! HT(\pi_v,\Pi_t))$ sont nuls
pour $i < 0$, on doit nÈcessairement avoir $k=0$ et donc $Q_0$ est nul. Soit $A$ le poussÈ en
avant
$$\xymatrix{
0 \ar[r] & P_0 \oplus Q_0 \ar[r] \ar@{->>}[d] & j^{\geq tg}_! HT(\pi_v,\Pi_t) \ar[r] \ar@{-->>}[d] &
j^{\geq tg}_{!*} HT(\pi_v,\Pi_t) \ar[r] \ar@{=}[d] & 0 \\
0 \ar[r] & P_0 \ar@{-->}[r] & A \ar[r] &  j^{\geq tg}_{!*} HT(\pi_v,\Pi_t) \ar[r]
& 0
}$$
de sorte que l'on a la suite exacte courte
$0 \rightarrow Q_0 \longrightarrow  j^{\geq tg}_! HT(\pi_v,\Pi_t)  \longrightarrow A \rightarrow 0$
o˘ la flËche $A \rightarrow Q_0[1]$ se factorise par $j^{\geq tg}_{!*} HT(\pi_v,\Pi_t)
\rightarrow Q_0[1]$. Pour $k$ non nul, par puretÈ on en dÈduit alors que les flËches
$H^{i}(A) \rightarrow H^{1+i}(Q_0)$ sont nulles et donc
$H^i ( j^{\geq tg}_! HT(\pi_v,\Pi_t)) \simeq H^i(A) \oplus H^i(Q_0).$
Il suffit alors de prouver que $H^{-k}(A)$ est non nul. 
Soit alors $\Pi$ la reprÈsentation automorphe du paragraphe prÈcÈdent avec donc
$\Pi_v \simeq [\overrightarrow{s-1}]_{\pi_v}$. 
D'aprËs le lemme \ref{lem-rappel-coho}, 
$\Pi_t  \overrightarrow{\times} LT_{\pi_v}(s-t,k-1) \otimes \Xi^{\frac{s-t}{2}}$ est un 
constituant de $[H^{-k}(j^{\geq (s-k)g}_{!*} HT(\pi_v,\Pi_t  \overrightarrow{\times} 
\st_{s-t-k}(\pi_v))\otimes \Xi^{-\frac{s-t-k}{2}}] \{ \Pi^{\oo,v} \}$ et d'aucun autre des constituants 
irrÈductibles de $A$ de sorte que, via la suite spectrale associÈe ‡ la filtration par les
poids de $A$, $\Pi_t  \overrightarrow{\times} LT_{\pi_v}(s-t,k-1) \otimes \Xi^{\frac{s-t}{2}}$
est un constituant de $H^{-k}(A)[\Pi^{\oo,v}]$ qui est donc non nul, d'o˘ le rÈsultat.

\rem la preuve du thÈorËme \ref{theo-boyer} dans \cite{boyer-invent2} procËde 
par rÈcurrence sur $t$ de $s$ ‡ $1$ comme suit.
La suite spectrale $E_1^{p,q}(P)=\hi^{p+g} \gr^{-p}(P) \Rightarrow \hi^{p+q} P$ 
associÈe ‡ la filtration par les poids de $P$ dÈgÈnËre en $E_1$ et le gros du travail consiste
‡ montrer que pour tout $0 \leq p \leq s-t-3$, la flËche $d_1^{p,t+1-s}$ induit un morphisme
non nul
$$\Bigl (\Pi_t \overrightarrow{\times} \pi_v \Bigr ) \overrightarrow{\times} \st_p(\pi_v)
\overrightarrow{\times} \speh_{s-t-p-1}(\pi_v) \longrightarrow 
\Bigl ( \Pi_t \overrightarrow{\times} \pi_v \Bigr ) \overrightarrow{\times} 
\st_{p+1}(\pi_v)  \overrightarrow{\times} \speh_{s-t-p-2}(\pi_v).$$
Avec les notations prÈcÈdentes, la surjection
$j^{\geq (t+1)g}_! HT(\pi_v, \Pi_t \overrightarrow{\times} \pi_v) \otimes
\Xi^{-1/2} \twoheadrightarrow P$
fournit par fonctorialitÈ un morphisme entre leurs filtrations exhaustives de stratification
et donc des diagrammes commutatifs
$$\xymatrix{
E_1^{p,q}(!) \ar@{->>}[d] \ar[rr]^{d_1^{p,q}(!)}  & & E_1^{p+1,q}(!) \ar@{->>}[d] \\
E_1^{p,q}(P) \ar[rr]_{d_1^{p,q}(P)} & & E_1^{p+1,q}(P),
}$$
o˘ les $E_1^{p,q}(!)$ sont les termes initiaux de la suite spectrale calculant les
faisceaux de cohomologie de 
$j^{\geq (t+1)g}_! HT(\pi_v, \Pi_t \overrightarrow{\times} \pi_v) \otimes \Xi^{-1/2}$.
La non nullitÈ des $d_1^{\delta,t+1-s}(P)$ pour $(p,q)=(\delta,t+1-s)$ avec 
$0 \leq \delta <s-t-1$, dÈcoule de celle Èvidente des $d_1^{\delta,t+1-s}(!)$.

\begin{figure}[ht]
\centering
\input{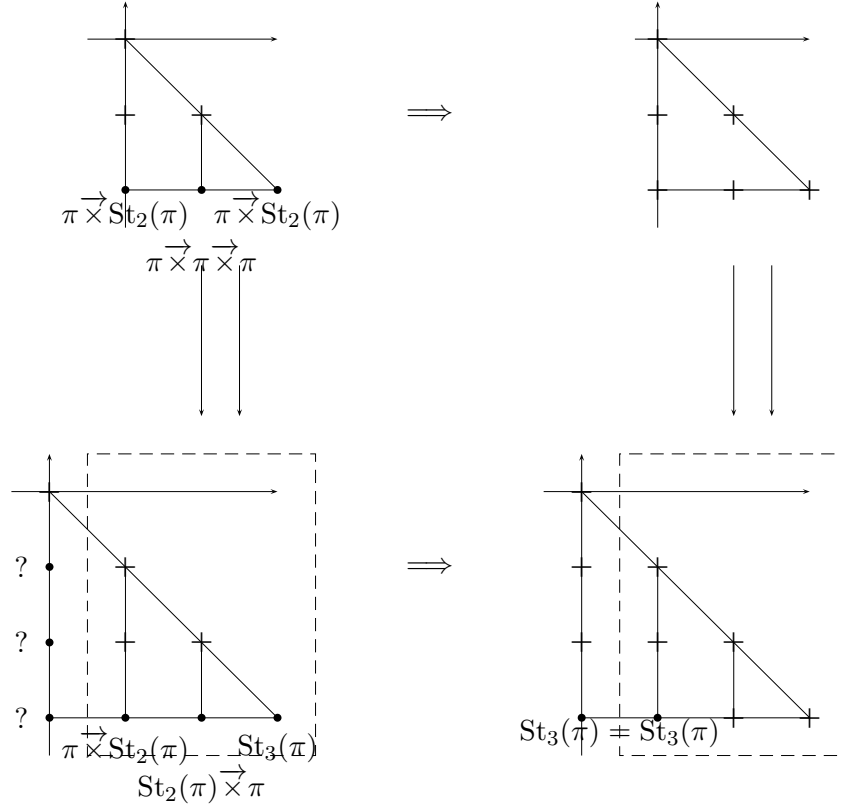}
\caption{\label{fig-ss} Germes en un point supersingulier spectrales 
$E_1^{p,q}(!) \longrightarrow E_1^{p,q}(P)$; on explicite simplement l'action du groupe linÈaire en \og factorisant
\fg{} par $\Pi_t \protect \overrightarrow{\times}$, i.e. quand on Ècrit
$\pi \protect \overrightarrow{\times} \protect \st_2(\pi)$ il faut lire 
$\Pi_t  \protect \overrightarrow{\times}
\protect \pi_v \protect \overrightarrow{\times}
\protect \st_2(\pi_v)$.}
\end{figure}

\subsection{Retour sur les filtrations de stratification de \texorpdfstring{$\Psi_{\IC,\pi_v}$}{Lg}}
\label{para-psi}

Reprenons la fin de preuve de la proposition \ref{prop-psi-fil} o˘, en raisonnant par l'absurde,
nous avons considÈrÈ $k$ tel que $\Fil^{kg}_!(\Psi_{\IC,\pi_v})$ 
soit strictement infÈrieure ‡ 
$\sum_{t=1}^k \sum_{i=t}^s \PC(i,\pi_v)(\frac{1+i-2t}{2})$ et 
$i_0>k$ minimal tel que 
$\PC(i_0,\pi_v)(\frac{1+i-2t}{2})$  ne soit pas un  constituant de $\Fil^{kg}_!(\Psi_{\IC,\pi_v})$. 
En raisonnant par rÈcurrence sur la cohomologie des espaces de Lubin-Tate, on se ramËne
au cas o˘ $i_0=s$.

Soit alors $\Pi$ la reprÈsentation automorphe du \S \ref{para-rappel-coho} avec donc
$\Pi_v \simeq \speh_s(\pi_v)$.
Notons tout d'abord que la fibre en un point supersingulier de $\hi^0 \Psi_{\IC,\pi_v}$ admet
$\pi_v[s]_D \otimes \st_s(\pi_v) \otimes L_g(\pi_v) (\frac{s+1-2t}{2})$ 
comme sous-quotient de sorte que $[H^0(\hi^0 \Psi_\IC)] \{ \Pi^{\oo,v} \}$ admet 
$\st_s(\pi_v) \otimes L_g(\pi_v) (\frac{s+1-2t}{2})$ comme sous-quotient. 
D'aprËs la proposition \ref{prop-rappel-coho}, pour tout $i>0$, 
$[H^i(\hi^{1-i} \Psi_\IC)] \{ \Pi^{\oo,v} \}$ n'admet pas $\st_s(\pi_v) \otimes
L_g(\pi_v) (\frac{s+1-2t}{2})$ comme sous-quotient de sorte que $[H^0_\eta] \{ \Pi^{\oo,v} \}$
admet $\st_s(\pi_v) \otimes L_g(\pi_v) (\frac{s+1-2t}{2})$ comme sous-quotient,
ce qui contredit le fait que $[H^0_{\bar \eta} ] \{ \Pi^\oo \}$
est non nul si et seulement si $\Pi$ est une reprÈsentation automorphe.

Enfin le calcul des fibres des faisceaux de cohomologie de $\Psi_{\IC,\pi_v}$ de
\cite{boyer-invent2} procËde comme suit. On considËre la suite spectrale
$E_1^{p,q}=\hi^{p+q} gr^{-p}_{\pi_v} \Rightarrow \hi^{p+q} \Psi_{\IC,\pi_v}$,
laquelle d'aprËs le calcul du paragraphe prÈcÈdent, des faisceaux de cohomologie
des faisceaux pervers d'Harris-Taylor, et donc des $E_1^{p,q}$, dÈgÈnËre en $E_1$.
Le gros du travail de \cite{boyer-invent2} consiste ‡ montrer qu'en un point supersingulier
$z$, pour tout $0 \leq \delta < s-1$ et $-\delta \leq p \leq s-1-2 \delta$, 
la flËche $d_1^{p,1-s+2\delta}$ induit un morphisme non nul
$$\st_{p+2\delta+1}(\pi_v) \overrightarrow{\times}  \speh_{s-p-2\delta}(\pi_v)
\longrightarrow \st_{p+2\delta+2}(\pi_v) \overrightarrow{\times} \speh_{s-p-2\delta-1}(\pi_v)
.$$
DÈsormais la non nullitÈ de cette flËche dÈcoule, comme dans le paragraphe prÈcÈdent, 
de la surjectivitÈ du morphisme d'adjonction $j^{= tg}_! j^{= tg} \gr^{tg}_! (\Psi_{\IC,\pi_v})
\longrightarrow \gr^{tg}_! (\Psi_{\IC,\pi_v})$.

\bibliographystyle{plain}
\bibliography{bib-ok}

\def\cftil#1{\ifmmode\setbox7\hbox{$\accent"5E#1$}\else
  \setbox7\hbox{\accent"5E#1}\penalty 10000\relax\fi\raise 1\ht7
  \hbox{\lower1.15ex\hbox to 1\wd7{\hss\accent"7E\hss}}\penalty 10000
  \hskip-1\wd7\penalty 10000\box7} \def\cprime{$'$}
\begin{thebibliography}{10}

\bibitem{sga43}
{\em Th\'eorie des topos et cohomologie \'etale des sch\'emas. {T}ome 3}.
\newblock Lecture Notes in Mathematics, Vol. 305. 1973.
\newblock (SGA 4), Dirig{\'e} par M. Artin, A. Grothendieck et J. L. Verdier.
  Avec la collaboration de P. Deligne et B. Saint-Donat.

\bibitem{ast}
A.~A. Beilinson, J.~Bernstein, and P.~Deligne.
\newblock Faisceaux pervers.
\newblock In {\em Analysis and topology on singular spaces, I (Luminy, 1981)},
  volume 100 of {\em Ast\'erisque}, pages 5--171. Soc. Math. France, Paris,
  1982.

\bibitem{boyer-invent2}
P.~Boyer.
\newblock Monodromie du faisceau pervers des cycles \'evanescents de quelques
  vari\'et\'es de {S}himura simples.
\newblock {\em Invent. Math.}, 177(2):239--280, 2009.

\bibitem{boyer-compositio}
P.~Boyer.
\newblock Cohomologie des systËmes locaux de {H}arris-{T}aylor et applications.
\newblock {\em Compositio}, 146(2):367--403, 2010.

\bibitem{Del-finitude}
P.~Deligne.
\newblock ThÈorËmes de finitude en cohomologie $\ell$ adique.
\newblock In {\em SGA 4$\frac{1}{2}$}.

\bibitem{fargues-faltings}
L.~Fargues, A.~Genestier, and V.~Lafforgue.
\newblock {\em {L'isomorphisme entre les tours de Lubin-Tate et de Drinfeld.}}
\newblock {Progress in Mathematics 262. Basel: Birkh\"auser}, 2008.

\bibitem{h-t}
M.~Harris, R.~Taylor.
\newblock {\em The geometry and cohomology of some simple {S}himura varieties},
  volume 151 of {\em Annals of Mathematics Studies}.
\newblock Princeton University Press, Princeton, NJ, 2001.

\bibitem{ill}
L.~Illusie.
\newblock Autour du thÈorËme de monodromie locale.
\newblock In {\em PÈriodes $p$-adiques}, number 223 in AstÈrisque, 1994.

\bibitem{ito2}
T.~Ito.
\newblock Hasse invariants for somme unitary {S}himura varieties.
\newblock {\em Math. Forsch. Oberwolfach report 28/2005}, pages 1565--1568,
  2005.

\bibitem{juteau}
D.~Juteau.
\newblock Modular {S}pringer correspondence and decomposition matrices.
\newblock {\em ThËse de l'UniversitÈ Paris 7}, 2007.

\bibitem{quasi-ab}
J.-P. Schneiders.
\newblock Quasi-abelian categories and sheaves.
\newblock {\em M\'emoires de la SMF 2e s\'erie, tome 76}, 1999.

\bibitem{zelevinski2}
A.~V. Zelevinsky.
\newblock Induced representations of reductive {${p}$}-adic groups. {II}. {O}n
  irreducible representations of {${\rm GL}(n)$}.
\newblock {\em Ann. Sci. \'Ecole Norm. Sup. (4)}, 13(2):165--210, 1980.

\end{thebibliography}

\end{document}